\documentclass[10pt]{amsart}
\usepackage[letterpaper,margin=1in]{geometry}
\usepackage{url}
\usepackage{amssymb,amsfonts,enumitem,mathtools,mathrsfs,hyperref}
\usepackage{algorithm,algpseudocode}
\usepackage{graphicx,caption,subcaption} 
\usepackage{tikz}
\usepackage[utf8]{inputenc}
\usepackage[normalem]{ulem}
\usepackage{multirow, booktabs}
\theoremstyle{plain}

\newtheorem{assumption}{Assumption}[section]
\newtheorem{proposition}{Proposition}[section]
\newtheorem{theorem}{Theorem}[section]
\newtheorem{definition}[theorem]{Definition}

\newtheorem{lemma}[theorem]{Lemma}
\newtheorem{remark}[theorem]{Remark}
\newtheorem{rulenv}{Rule}[section]
\numberwithin{equation}{section}

\makeatletter
\AtBeginDocument{\mathchardef\xd@origx=\mathcode`\x \mathcode`\x="8000 }
\begingroup
  \lccode`\~=`\x
  \lowercase{\endgroup\def~}{\bgroup\@ifnextchar^\xd@check{\xd@origx\egroup}}
\def\xd@deltatok{\delta}
\def\xd@check^#1{%
  \def\xd@arg{#1}%
  \ifx\xd@arg\xd@deltatok
    w\egroup^{\delta}%
  \else
    \xd@origx\egroup^{#1}%
  \fi}
\makeatother

\begin{document}
\title[D-STAR method]{Data-driven stochastic asymptotic regularization for ill-posed inverse problems with a posteriori and heuristic stopping}
\author[S. Soni]{Simran Soni}
\author[H. Bajpai]{Harshit Bajpai}
\author[A. K. Giri]{Ankik Kumar Giri}
\thanks{
\hspace{-0.64cm} Department of Mathematics, Indian Institute of Technology Roorkee, Roorkee, Uttarakhand 247667, India.\\[3pt]
E-mail: \texttt{simran\_s@ma.iitr.ac.in} (Simran Soni), \texttt{harshit\_b@ma.iitr.ac.in} (Harshit Bajpai), \texttt{ankik.giri@ma.iitr.ac.in} (Ankik Kumar Giri).\\[3pt]
 Funding: S. Soni and H. Bajpai are supported by the Ministry of Education, Government of India, and
 the Indian Institute of Technology Roorkee.
}
\maketitle
\begin{quote}
{{\em \bf Abstract.}}
A novel Data-driven STochastic Asymptotic Regularization (\texttt{D-STAR}) method is proposed for solving linear and nonlinear ill-posed inverse problems, which incorporates a data-driven drift and a stochastic diffusion mechanism. Both known and unknown noise level cases are accommodated by the method, thereby making the framework applicable to a broader range of practical inverse problems.
For known noise level, a modified discrepancy principle is employed, whereas for unknown noise level, a heuristic stopping rule based solely on noisy data is developed. For each stopping rule, an appropriate diffusion term is defined and the well-posedness of the associated stochastic differential equation is demonstrated. The finite termination of \texttt{D-STAR} method under both stopping rules is established and its convergence to the exact solution in the noise-free setting is proved. For the noisy data, stability and the regularization property in the mean-square sense are further established. In addition, the uncertainty quantification of the proposed method and its comparison with the Stochastic Asymptotic Regularization (\texttt{SAR}) method under both rules are provided. Furthermore, it is shown that local minima arising in the Data-Driven Iteratively Regularized Landweber (\texttt{DDIRLI}) method may be escaped. Finally, detailed numerical experiments, including comparisons with existing \texttt{DDIRLI} and \texttt{SAR} methods, are presented on X-ray computed tomography and nonlinear schlieren tomography for both stopping rules, thereby demonstrating the performance of \texttt{D-STAR}.
\end{quote}
\vspace{.3cm}

\noindent

\begin{description}
\item[\textbf{Keywords}] nonlinear inverse problems, uncertainty quantification, asymptotic regularization, stochastic algorithm, data-driven regularization. \vspace{.1cm}
\item[\textbf{MSC codes}] 47A52, 65J20, 65J22, 82B31.
\end{description}


\section{Introduction}
\noindent In this article, we solve the following deterministic nonlinear ill-posed operator equation
\begin{align}
    \mathcal{F}(w)=v, \label{IP}
\end{align}
where $\mathcal{F}: \mathcal{D}(\mathcal{F}) \subset \mathcal{W} \to \mathcal{V}$ is a nonlinear operator between separable Hilbert spaces $\mathcal{W}$ and $\mathcal{V}$ with the usual inner products $\langle \cdot , \cdot \rangle$ and norms $\|\cdot\|$, respectively.
Such inverse problems have various applications including medical tomography~\cite{barrett1992image}, estimation of reaction rates in biological systems~\cite{shukla2020physics} and  parameter identification problem in PDEs~\cite{hanke1995convergence}.
In general, the solution to  \eqref{IP} may not exists due to its ill-posedness or even if exists, it may fail to be unique or depend continuously on perturbations of $v$. Although in this work, we assume that  \eqref{IP} admits a solution $w^\dagger$, which may not be unique. Our primary interest is in problems of the form \eqref{IP} where the solution $w^\dagger$ does not depend continuously on the data $v$. These problems require regularization methods to obtain reliable approximations of $w^\dagger$. In practical applications, only noisy measurement $v^\delta$ of the exact data $v$ is available. When the noise level $\delta>0$ is known and the noisy data $v^\delta$ satisfies 
\begin{align}\label{noise:lev}
    \|v^\delta -v\| \leq \delta,
\end{align}
several variational \cite{scherzer2009variational} and iterative regularization methods \cite{kaltenbacher2008iterative} have been proposed over the past decades. 
Variational methods recover the solution by solving an optimization problem, while iterative methods generate successive approximations and exhibit semi-convergence behavior due to ill-posedness of the problem. Thus, an appropriate stopping criterion is necessary to achieve a reliable solution. When the noise level $\delta$ is known, discrepancy principle~\cite{morozov1966solution} guarantees the regularization. An iterative Landweber iteration in this setting is  one of the most prominent and extensively studied method, which is given by 
\begin{align}
    w_{k+1}^\delta = w_k^\delta -\Delta t\mathcal{F}'(w_k^\delta)^* [\mathcal{F}(w_k^\delta)-v^\delta], \quad k \in \mathbb N,  \label{LI}
\end{align}
where $w_0^\delta = \tilde{w}$ is an initial value, $\Delta t$ is step size
and $\mathcal{F}'(w_k^\delta)^*$ represents the adjoint of Fr\'echet derivative of $\mathcal{F}(w_k^\delta)$.
The continuous counterpart of \eqref{LI}, when $\Delta t$ tends to zero is known as the asymptotical regularization or
Showalter\textquotesingle{s} method \cite{tautenhahn1994asymptotical}, given by
\begin{align}
    dx^\delta(t) = \mathcal{F}'(x^\delta(t))^* [v^\delta-\mathcal{F}(x^\delta(t))]dt, \label{AR}
\end{align}
where $t \geq 0$ is an artificial scalar time. 
It was shown in \cite{rieder2005runge,zhao2020convergence} that Runge--Kutta integrators preserve the regularization properties of \eqref{AR} in its numerical realization. In \cite{boct2022convergence,gong2020new,zhang2019fractional,zhang2020second}, the authors proposed the fractional-order and second-order gradient flows to accelerate the convergence of \eqref{AR}. 
Moreover, in order to preserve special structures of solution for nonlinear inverse problems such as sparsity and discontinuities, an  asymptotical regularization with convex constraints is studied in \cite{zhong2022asymptotical}. Furthermore, stochastic extensions have also been considered, in particular, asymptotic regularization for linear inverse problems with additive white noise is analyzed in \cite{lu2021asymptotical}.

It is worth to emphasize that Landweber method \eqref{LI} has local convergence for nonlinear ill-posed problems and it may become trapped in the local minima for an optimization problem, $\min \|\mathcal{F}(w)-v^\delta\|^2$. Therefore, to escape such local minima and to achieve global convergence, stochastic perturbations have been incorporated into \eqref{LI}. The resulting stochastic  method is given by
\begin{align}
    w_{k+1}^\delta = w_k^\delta -\Delta t \mathcal{F}'(w_k^\delta)^* [\mathcal{F}(w_k^\delta)-v^\delta] + g_k \Delta B_k, \quad k \in \mathbb N,  \label{SM}
\end{align}
where $x^\delta_0=\tilde{w}$ is a deterministic initial value, $\{\Delta B_k\}_{k \geq 1}$ denotes a sequence of random variables with a specified distribution, and $g_k$ determines the amplitude of the stochastic perturbation. According to the numerical optimization \cite{chiang1987diffusion,engquist2022algebraically,gelfand1991recursive,kushner1987asymptotic}, $g_k$ should gradually vanish to guarantee the convergence. However, if $g_k$ decays too rapidly, \eqref{SM} may converge to a local minima rather than the desired global minima. Therefore, selecting an appropriate sequence $g_k$ is crucial. The continuous analog to \eqref{SM} is given by
\begin{align}
    dx^\delta(t) = \mathcal{F}'(x^\delta(t))^* [v^\delta-\mathcal{F}(x^\delta(t))]dt + g(t)dB(t), \label{SAR}
\end{align}
where $t \geq 0$, $g(t) \in L^\infty(\mathbb{R}_+)$ and $\{B(t)\}_{t \geq 0}$ is a $\mathcal W$-valued  $\mathcal{Q}$-Wiener process. Equation \eqref{SAR} is known as the {Stochastic Asymptotic Regularization} (\texttt{SAR}) method.

The convergence of \texttt{SAR} method \eqref{SAR} has been studied under various stopping rules. For linear ill-posed inverse problems with known noise level, both \textit{a priori} and \textit{a posteriori} stopping rules have been considered in \cite{zhang2023stochastic}, whereas a random discrepancy principle is introduced in \cite{zhang2024stochastic}. In the nonlinear setting with known noise level, the convergence results for \texttt{SAR} method using a modified Morozov discrepancy principle is established in \cite{long2025stochastic}. However, to the best of our knowledge, the convergence of \texttt{SAR} method under stopping rules that do not require the knowledge of noise level remains unexplored. It is important to highlight that
\texttt{SAR} method possesses four main advantages over deterministic regularization methods. By incorporating stochastic perturbations, \texttt{SAR} helps to escape from local minima by selecting an optimal path. Additionally, it  provides uncertainty quantification of approximate solution of \eqref{IP}. Furthermore, \texttt{SAR} can uncover hidden information in many real-world problems that may remain inaccessible due to incomplete mathematical models and large-scale observational noise, see \cite[subsection 5.2]{zhang2023stochastic} for an example in biosensor tomography. Lastly \texttt{SAR} can identify multiple solutions by clustering the obtained approximate solution.  

It is worth emphasizing that \texttt{SAR} method is fundamentally distinct from the Stochastic Gradient Descent (\texttt{SGD}) method \cite{robbins1951stochastic}, which is widely used in  machine learning, statistics and computational mathematics. For comprehensive convergence analysis of \texttt{SGD} methods to solve ill-posed inverse problems, we refer the reader to \cite{jahn2020discrepancy,jin2020convergence,jin2021saturation,lu2022stochastic,zhang2024stochastic}. We note \texttt{SGD} is inherently a discrete iterative method while \texttt{SAR} is formulated as a continuous-time stochastic process. Further, \texttt{SGD} can be view as a randomized variant of the classical Landweber iteration but it differs from \texttt{SAR} in both the origin of randomness and their computational cost.
In \texttt{SGD}, randomness is introduced by randomly selecting a data sample or row index at each iteration, thereby avoiding the computation of full gradient and significantly reducing the computational cost.
On the other hand, \texttt{SAR} incorporates randomness through the stochastic perturbation $g(t)dB(t)$ which is introduced to escape from the local minima and quantifying uncertainties rather than computational efficiency.

In the discrete setting, when \textit{a priori} information $w^{(0)}$  about the true solution is available, then such information can be effectively incorporated to improve the stability. Motivated by this idea, by adding a damping term to \eqref{LI}, the Iteratively Regularized Landweber Iteration (IRLI) \cite{scherzer1998modified} is proposed as
\begin{align}
    w_{k+1}^\delta = w_k^\delta -\Delta t\mathcal{F}'(w_k^\delta)^* [\mathcal{F}(w_k^\delta)-v^\delta] -\lambda_k^\delta (w_k^\delta-w^{(0)}), \quad k \in \mathbb N, \label{IRLI}
\end{align}
where $\lambda_k^\delta \geq 0$ is a damping parameter.
It has been shown that the inclusion of the damping term $-\lambda_k^\delta (w_k^\delta-w^{(0)})$  stabilizes the method without needing additional assumptions, see \cite{kaltenbacher2008iterative}. Building on \texttt{IRLI} method \eqref{IRLI}, the Data-Driven  Iteratively Regularized Landweber Iteration  (\texttt{DDIRLI}) is proposed in \cite{aspri2020data}.
The main aim of \texttt{DDIRLI} method is to include a damping factor trained from a prior data $(w_i,\mathcal{F}(w_i))_{1 \leq i \leq m}$. To be more specific, define a bounded linear operator $\hat{A}$ between Hilbert spaces $\mathcal{W}$ and $\mathcal{V}$, i.e. $\hat{A} \in \mathcal{B}(\mathcal{W},\mathcal{V})$ such that $\hat{A}(w_i):= \mathcal{F}(w_i)$. Then, the resulting method is given by
\begin{align}
    w_{k+1}^\delta = w_k^\delta -\Delta t\mathcal{F}'(w_k^\delta)^* [\mathcal{F}(w_k^\delta)-v^\delta] - \lambda_k^\delta \hat{A}'(w_k^\delta)^* \big[\hat{A}(w_k^\delta)-v^\delta \big], \quad k \in \mathbb N, \label{DDLI}
\end{align}
where $\hat{A}'(w_k^\delta)^*$ represents the adjoint of Fr\'echet derivative of $\hat{A}(w_k^\delta)$. The additional data-driven term incorporates information learned from the training data. At the beginning of the iterations, this term has a greater influence on the reconstruction, while its effect gradually decreases as the iteration proceeds. Consequently, the reconstruction is increasingly governed by the operator $\mathcal{F}$.
 Moreover, when the desired solution is well approximated by the training set, \texttt{DDIRLI} method \eqref{DDLI} can give better reconstruction accuracy than the Landweber iteration \eqref{LI} and  \texttt{IRLI} method~\eqref{IRLI}. For further details on data-driven regularization methods, we refer the reader to \cite{aspri2020dataproj,aspri2023analysis,bajpai2025stochastic,gao2023fast,haltmeier2024data,tong2023data,zhou2025data}. 

Although \texttt{DDIRLI} method \eqref{DDLI} effectively incorporates a prior information from training data, but it inherits the limitations of classical iterative regularization  methods due to its deterministic nature. In contrast, \texttt{SAR} method \eqref{SAR} provides the uncertainty quantification and probabilistic information about the recovered solution, but it does not utilize the prior knowledge learned from training data. This motivates incorporating the data-driven strategy into the stochastic asymptotic regularization framework to combine the strengths of both approaches.
Thus, we propose the following \textit{Data-driven STochastic Asymptotic Regularization} (\texttt{D-STAR}) method as
\begin{align}\label{D-STAR}
    dx^{\delta}(t) = \big(\mathcal{F}' (x^{\delta}(t))^* [v^{\delta} - \mathcal{F}(x^{\delta}(t)) ] + \lambda^{\delta}(t) \hat{A}'(x^{\delta}(t))^* [v^{\delta}-\hat{A}(x^{\delta}(t))] \big)  dt + g^{\delta}(t,x^{\delta}(t)) dB(t),
\end{align}
where $t \geq 0$, $x^\delta(0) = \tilde{w} \in \mathcal{W}$ is the deterministic initial value, $\lambda^{\delta}(t)$  denotes the regularization weight with $\sup_{t \geq 0} \lambda^{\delta}(t) < \Gamma$, $\hat{A} \in \mathcal{B}(\mathcal{W},\mathcal{V})$ and $\{B(t)\}_{t \geq 0}$ is $\mathcal{W}$-valued $\mathcal{Q}$-Wiener process. 
A detailed discussion of $\Gamma$ and $g^{\delta}(t,x^{\delta}(t))$ for known and unknown noise levels are presented in Section \ref{sec:MDP} and Section \ref{sec:heu}, respectively. Moreover, the  construction of $\hat{A}$ follows from ~\cite{aspri2020data}.

We establish the detailed convergence analysis of \texttt{D-STAR} method \eqref{D-STAR} using a modified discrepancy principle which guarantees the finite termination of method.
Compared with the analysis of \texttt{DDIRLI} method in \cite{aspri2020data}, the present setting involves additional analytical challenges. In \cite{aspri2020data}, the regularization method is deterministic and the stopping index is defined in a discrete setting, whereas our proposed method is stochastic and the stopping criterion is formulated in continuous setting.
As a result, the deterministic and discrete  arguments developed therein cannot be applied here directly. Instead, the analysis must be carried out within a probabilistic framework and continuous setting.
The present work also differs from \cite{long2025stochastic} because of the incorporation of the data-driven regularization term into the \texttt{SAR} method. This additional term introduces new difficulties in establishing the well-posedness of \texttt{D-STAR} method \eqref{D-STAR}. While the stochastic component already necessitates the use of  It\^o\textquotesingle{s} formula, an inclusion of the data-driven term produces an additional product term between the classical Landweber and the data-driven components. 
Controlling this interaction constitutes a major technical challenge in the convergence analysis.

It is important to emphasize that the modified discrepancy principle requires prior knowledge of the noise level, which is unavailable in various real-world applications.
This limitation leads to the development of {heuristic rules} \cite{bajpai2026hanke,jahn2023noise,jin2017heuristic,kindermann2020convergence,real2024hanke}, where the stopping criterion is determined by noisy data $v^\delta$ only. To address this setting, we further establish a detailed convergence analysis of the \texttt{D-STAR} method \eqref{D-STAR} under a heuristic stopping rule. This analysis presents additional challenges, which arises due to presence of data-driven and stochastic diffusion terms. In particular, the diffusion coefficient needs to be suitably constructed and controlled without an explicit knowledge of the noise level, while its interaction with the data-driven component must also be handled within the stochastic framework.
To the best of our knowledge, this work provides the first convergence analysis under a heuristic stopping rule within a data-driven, continuous-time, or \texttt{SAR} framework.

We summarize the main contributions of this work as follows:
\begin{itemize}
    \item This is the first rigorous convergence analysis of a regularization method that combines the stochastic asymptotic regularization with the data-driven priors for solving linear and nonlinear ill-posed inverse problems. The proposed framework inherits the stability and regularization properties of stochastic asymptotic regularization while simultaneously exploiting prior information learned from training data to improve the reconstruction.
    
    \item  We establish the  detailed convergence analysis of \texttt{D-STAR} method under a modified discrepancy principle for known noise level, including finite termination, exact data case convergence, and stability along with regularization behavior for noisy data.

   \item We extend the convergence analysis of \texttt{D-STAR} method for unknown noise level by employing a heuristic stopping criterion based only on the noisy data. It constitutes a novel aspect of this work as convergence analysis under a heuristic type stopping has not been established previously even for \texttt{DDIRLI} or \texttt{SAR}, or within a continuous-time framework.
    
    \item We provide numerical experiments to validate \texttt{D-STAR} method and compare its performance with \texttt{DDIRLI} and \texttt{SAR}. The numerical study illustrates the reconstruction accuracy, stopping behavior, and uncertainty quantification under different noise levels and both stopping strategies.
\end{itemize}
\noindent The paper is organized as follows. Section \ref{sec:pre} recalls the necessary concepts from stochastic calculus. Section \ref{sec:as} introduces the required assumptions for the convergence analysis and presents two stopping strategies: a modified discrepancy principle and a heuristic stopping rule. Sections \ref{sec:MDP} and \ref{sec:heu} are devoted for the convergence analysis of \texttt{D-STAR} method under a modified discrepancy principle and a heuristic stopping rule, respectively. Section \ref{sec:num} illustrates the numerical experiments to validate the proposed method and compare its reconstruction accuracy, convergence behavior, and uncertainty quantification with \texttt{DDIRLI} and \texttt{SAR}. Finally, Section \ref{sec:con} concludes this work and gives the possible future directions.


\section{Preliminaries} \label{sec:pre}

\noindent This section presents the fundamental concepts of stochastic calculus required for the analysis of \texttt{D-STAR} method ~\eqref{D-STAR}, whose formulation involves a stochastic integral. For the detailed discussion of these concepts, see~\cite{GawareckiMandrekar2011}. We begin by introducing the following definitions, which will be used in subsequent analysis.

\begin{definition}[Trace-class operator]
Let $\mathcal{W}$ be a real separable Hilbert space equipped with an orthonormal basis $\{ e_j\}_{j=1}^\infty$ and $T: \mathcal{W} \to \mathcal{W}$ be a bounded linear operator. Then $T$ is called a trace-class operator if $\text{tr} \big((TT^*)^{1/2} \big) < \infty$, where the trace of operator $\hat{T}:=(TT^*)^{1/2}$ is defined as $\text{tr}(\hat{T}) = \sum_{j=1}^\infty \langle \hat{T} e_j,e_j \rangle$.
\end{definition}

\begin{definition}[$\mathcal{W}$-valued $\mathcal{Q}$-Wiener process]

Let $\mathcal{Q}$ be a non-negative definite symmetric trace-class operator defined on a real separable Hilbert space $\mathcal{W}$. Suppose $\{\lambda_j\}_{j=1}^\infty$ denotes the eigen values of $\mathcal{Q}$ corresponding to eigen vectors $\{ e_j\}_{j=1}^\infty$, which are forming an orthonormal basis for $\mathcal{W}$. Further, let $\{\beta_j (t)\}_{t \geq 0, j \in \mathbb N}$ be a sequence of independent Brownian motions on filtered probability space $(\Omega,\mathscr{F},\{\mathscr{F}_t\}_{t\geq 0}, \mathbb P)$. Then, the $\mathcal{W}$-valued $\mathcal{Q}$-Wiener process is defined as
\begin{align*}
    B(t) = \sum_{j=1}^\infty \lambda_j^{1/2} \beta_j (t) e_j.
\end{align*}
\end{definition}
\begin{definition}[Hilbert–Schmidt operator]
Let $\mathcal{W}$ and $\mathcal{V}$ be two real separable Hilbert spaces and $\{ e_j\}_{j=1}^\infty$ be an orthonormal basis of $\mathcal{W}$, then a bounded linear operator $T$ is said to be a Hilbert–Schmidt operator if\, $\sum_{j=1}^\infty \|Te_j\|^2_{\mathcal{V}} < \infty$. The collection of all such Hilbert–Schmidt operators is denoted by
$\mathcal{L}_2(\mathcal{W},\mathcal{V})$. Since $\mathcal{W}$ and $\mathcal{V}$ are separable, so $\mathcal{L}_2(\mathcal{W},\mathcal{V})$ is also a separable Hilbert space with the norm 
\begin{align*}
    \|T\|_{\mathcal{L}_2(\mathcal{W},\mathcal{V})} = \Big( \sum_{j=1}^\infty \|Te_j\|^2_{\mathcal{V}} \Big)^{1/2}.
\end{align*}
\end{definition}

\noindent
We next present the construction of the It\^o\textquotesingle{s} stochastic integral with respect to a $\mathcal{W}$-valued $\mathcal{Q}$-Wiener process. 

Let $\mathcal{W}$ be a separable Hilbert space and $\mathcal{Q}$ be a non-negative definite symmetric trace-class operator on $\mathcal{W}$. Let $\{\lambda_j\}_{j=1}^\infty$ denotes the positive eigenvalues of $\mathcal{Q}$ and $\{ e_j\}_{j=1}^\infty$ be the corresponding eigenvectors, which form an orthonormal basis of $\mathcal{W}$. Then, the space $\mathcal{W}_{\mathcal{Q}} = \mathcal{Q}^{1/2} \mathcal{W}$ endowed with
\begin{align*}
    \langle x,y\rangle_{\mathcal{W}_{\mathcal{Q}}} = \sum_{j=1}^\infty \frac{1}{\lambda_j} \langle x,e_j \rangle_{\mathcal{W}} \langle y,e_j \rangle_{\mathcal{W}}, \quad x,y \in \mathcal{W}
\end{align*}
is a separable Hilbert space with an  orthonormal basis $\{\lambda_j^{1/2}e_j\}_{j=1}^\infty$.
Moreover, let $\mathcal{L}_2(\mathcal{W}_{\mathcal{Q}},\mathcal{V})$ denotes the space of Hilbert–Schmidt operators and $\{ \tilde{e}_i\}_{i=1}^\infty$ be an orthonormal basis in $\mathcal{V}$. Then, the Hilbert–Schmidt norm of  $T \in \mathcal{L}_2(\mathcal{W}_{\mathcal{Q}},\mathcal{V})$ is given by
\begin{align*}
    \|T\|^2_{\mathcal{L}_2(\mathcal{W}_{\mathcal{Q}},\mathcal{V})} = & \sum_{j=1}^\infty \|T(\lambda_j^{1/2}e_j)\|^2_{\mathcal{V}} = \sum_{i,j=1}^\infty \|\langle T(\lambda_j^{1/2}e_j), \tilde{e}_i \rangle\|^2_{\mathcal{V}} 
    = \sum_{i,j=1}^\infty \|\langle T(Q^{1/2}e_j),\tilde{e}_i \rangle\|^2_{\mathcal{V}} 
    \\
    = & \|TQ^{1/2}\|^2_{\mathcal{L}_2(\mathcal{W},\mathcal{V})}
    = \text{tr} \Big( (TQ^{1/2}) (TQ^{1/2})^* \Big).
\end{align*}
Additionally, on the probability space $(\Omega,\mathscr{F},\{\mathscr{F}_t\}_{t\geq 0}, \mathbb P),$ we define the space of $\mathcal{L}_2(\mathcal{W}_{\mathcal{Q}},\mathcal{V})$-valued process as
\begin{align*}
    \Lambda_2 (\mathcal{W}_{\mathcal{Q}},\mathcal{V}) = \Big\{ \psi : [0,T] \times \Omega \to \mathcal{L}_2(\mathcal{W}_{\mathcal{Q}},\mathcal{V}) ~ \text{such that}~ \mathbb E \Big[\int_0^T \|\psi(t)\|^2_{\mathcal{L}_2(\mathcal{W}_{\mathcal{Q}},\mathcal{V})} dt \Big] < \infty \Big\}
\end{align*} 
which is a Hilbert space with the norm
\begin{align*}
    \|\psi\|_{\Lambda_2 (\mathcal{W}_\mathcal{Q},\mathcal{V})} = \Big( \mathbb E \Big[\int_0^T  \|\psi(t)\|^2_{\mathcal{L}_2(\mathcal{W}_{\mathcal{Q}},\mathcal{V})} dt \Big]  \Big)^{1/2}.
\end{align*}
For $\psi \in \mathcal{L}_2(\mathcal{W}_{\mathcal{Q}},\mathcal{V})$, the stochastic integral $\int_0^t \psi(s) dB(s), t \in [0,T]$, can be defined as in finite dimensional case. For the detailed discussion, we refer the reader to \cite[section 2.2]{GawareckiMandrekar2011}.

We now state a theorem establishing that the stochastic integral is a square-integrable martingale satisfying the It\^o isometry, which will be used for handling
stochastic integral.

\begin{theorem}\cite[Theorem 2.3]{GawareckiMandrekar2011}\label{th:pre;stoc-int}
The stochastic integral $\psi \to \int_0^t \psi(s)dB(s)$ with respect to a $\mathcal{W}$-valued $\mathcal{Q}$-Wiener process $B(t)$ is a continuous square integrable martingale and satisfies the following condition,
\begin{align*}
    \mathbb E \Big[ \Big\|\int_0^t \psi(s) dB(s) \Big\|^2_{\mathcal{V}} \Big] = \mathbb E  \Big[ \int_0^t \|\psi(s)\|^2_{\mathcal{L}_2(\mathcal{W}_{\mathcal{Q}},\mathcal{V})}ds   \Big] < \infty.
\end{align*}
\end{theorem}

\noindent
The following theorem introduces the It\^o\textquotesingle{s} formula for stochastic process driven by a $\mathcal{W}$-valued $\mathcal{Q}$-Wiener process, which provides a fundamental tool for analyzing functionals of a solution to stochastic differential equation. 

\begin{theorem}\cite[Theorem 2.9]{GawareckiMandrekar2011}\label{th:pre;ito}
Let $\mathcal{W}$ and $\mathcal{V}$ be two separable Hilbert spaces and $B(t)$ be $\mathcal{W}$-valued $\mathcal{Q}$-Wiener process. Suppose that a stochastic process $Y(t),\; t \in [0,T],$ is given by
\begin{align*}
    Y(t) = Y(0) + \int_0^t \phi(s) ds + \int_0^t \psi(s) dB(s),
\end{align*}
where $Y(0)$ is $\mathscr{F}_0$-measurable $\mathcal{V}$-valued random variable and $\phi(s)$ is $\mathcal{V}$-valued $\mathscr{F}_s$-measurable Bochner-integrable process on $[0,T]$ such that $\int_0^t \|\phi(s)\|_{\mathcal{V}} ds < \infty$ $\mathbb P$-almost surely.
\newline
\noindent
Assume that a function $G:[0,T] \times \mathcal{V} \to \mathbb R$ is continuous and its Fr\'echet partial derivatives $G_t,G_x,G_{xx}$ are continuous and bounded on bounded subset of $[0,T]\times \mathcal{V}$. Then, the following It\^o\textquotesingle{s} formula holds,
\begin{align*}
    G(t,Y(t)) &=  G(0,Y(0)) + \int_0^t \langle G_x(s,Y(s)),\psi(s)dB(s)\rangle_{\mathcal{V}} 
    \\
    & + \int_0^t \Big\{ G_t(s,Y(s)) + \langle G_x(s,Y(s)),\phi(s)\rangle_{\mathcal{V}} 
     + \frac{1}{2} \text{tr} \Big[G_{xx} (s,Y(s)) (\psi(s)Q^{1/2}) (\psi(s)Q^{1/2})^* \Big]  \Big\} ds
\end{align*}
$\mathbb P$-almost surely, for all $t \in [0,T]$.
\end{theorem}

\noindent
We conclude this subsection by introducing the following nonlinear stochastic differential equation
\begin{align}\label{SDE}
    dY = \big(H Y + F_1 (t,Y) \big)dt + F_2 (t,Y) dB(t), \quad Y(0) = \zeta, 
\end{align}
where $\zeta$ is $\mathscr{F}_0$-measurable $\mathcal{W}$-valued random variable. For fixed $T>0$, the coefficients $H, F_1$, $ F_2$, and their corresponding well-posedness result is given in Theorem \ref{th:pre:well}.  To state this result, we first introduce the notion of a mild solution to \eqref{SDE}.

\begin{definition}[Mild solution]
Let $Y(t), t \in [0,T],$ be a $\mathcal{W}$-valued, predictable  process. Then it is said to be a \textit{mild solution} of \eqref{SDE} if $\mathbb P \Big( \int_0^T \|Y(s)\|^2_{\mathcal{W}} ds < \infty \Big)=1$, and
\begin{align*}
    Y(t) = G(t) \zeta + \int_0^t G(t-s) F_1(s,Y(s)) ds + \int_0^t G(t-s) F_2(s,Y(s)) dB(s) 
\end{align*}
$\mathbb P\text{-almost surely},$ for any $t \in [0,T]$, where $G(t):=e^{tH}$.    
\end{definition}

\begin{theorem}\cite[section 7.1]{DaPratoZabczyk2014} \label{th:pre:well}
Assume that
\begin{enumerate}[ leftmargin=*]
    \item[(i)] For $t \geq 0$, $H$ is the generator of a $C_0$-semigroup $G(t)=e^{tH}$,  in $\mathcal{W}$. \vspace{0.15cm}
    \item[(ii)] The mappings $F_1 :[0,T]\times \Omega \times \mathcal{W} \to \mathcal{W}$, $F_2 :[0,T]\times \Omega \times \mathcal{W} \to \mathcal{L}_2(\mathcal{W}_{\mathcal{Q}},\mathcal{V})$ are measurable from $[0,T] \times \Omega \times \mathcal{W}$ into $(\mathcal{W},\mathscr{B}(\mathcal{W}))$ and $(\mathcal{L}_2(\mathcal{W}_{\mathcal{Q}},\mathcal{V}), \mathscr{B}(\mathcal{L}_2(\mathcal{W}_{\mathcal{Q}},\mathcal{V}))$, respectively , where $\Omega$ is sample space and $\mathscr{B}(\cdot)$ denotes the Borel $\sigma$-field associated with the corresponding space. \vspace{0.15cm}

    \item[(iii)] There exist constants $K, \bar{K}>0$ such that for all $w,\bar{w} \in \mathcal{W}, t \in [0,T],$ and $\omega \in \Omega$,
    \begin{align}
        \|F_1(t,\omega,w)-F_1(t,\omega,\bar{w})\|_{\mathcal{W}} + \|F_2(t,\omega,w)-F_2(t,\omega,\bar{w})\|_{\mathcal{L}_2(\mathcal{W}_{\mathcal{Q}},\mathcal{V})} &\, \leq K \|w-\bar{w}\|_{\mathcal{W}}, \label{eq:pre:lip}
        \\
        \|F_1(t,\omega,w)\|^2_{\mathcal{W}} + \|F_2(t,\omega,w)\|^2_{\mathcal{L}_2(\mathcal{W}_{\mathcal{Q}},\mathcal{V})} &\, \leq \bar{K}(1+\|w\|^2_{\mathcal{W}}). \label{eq:pre:gr}
    \end{align}
\end{enumerate}
Then, there exists a mild solution $Y(t)$ to \eqref{SDE} which is unique, up to equivalence, among the
processes satisfying $\mathbb P \Big( \int_0^T \|Y(s)\|^2_{\mathcal{W}} ds < \infty \Big)=1$. Moreover, the solution $Y(t)$ possesses a continuous modification and for any $p\geq2$, there exist constants $K_{p,T},\hat{K}_{p,T}>0$ such that 
\begin{align}
    \sup_{t \in [0,T]} \mathbb E \big[\|Y(t)\|^p_{\mathcal{W}} \big] \leq &\, K_{p,T} \big(1+\mathbb E \big[\|\zeta\|^p _{\mathcal{W}}\big] \big) \quad \text{and} \quad \mathbb E \left[ \sup_{t \in [0,T]} \|Y(t)\|^p_{\mathcal{W}} \right] \leq \hat{K}_{p,T} \big(1+\mathbb E \big[\|\zeta\|^p_{\mathcal{W}} \big] \big).\label{eq:pre:MB1}
\end{align}
    
\end{theorem}


\section{Assumptions and Stopping Rules}
\label{sec:as}
\noindent
This section begins by specifying the assumptions required to establish the convergence analysis of \texttt{D-STAR} method~\eqref{D-STAR}. For notational convenience, explicit reference to the underlying norm spaces is omitted throughout the convergence analysis. Subsequently, we present the stopping rules that govern the termination of the proposed method.

Let $w^\dagger$ be the solution of \eqref{IP} and $B_\rho(w^\dagger)$ denotes the ball of radius $\rho>0$ centered at $w^\dagger$, i.e. $B_\rho(w^\dagger):= \{w\in \mathcal{W} : \big(\mathbb E \big[\|w-w^\dagger\|^2 \big] \big)^{1/2} \leq \rho \}$. Then, we make the following assumptions on $\mathcal{F}$ and $\hat{A}$. 

\begin{assumption}\label{Assum}
Suppose there exists $w^\dagger \in B_\rho(\tilde{w})$ such that $\mathcal{F}(w^\dagger)=v$. Then,
\begin{enumerate}[label=($A\arabic*$)]
    \item $\mathcal{F}$ is a continuous operator and possess a continuous, bounded Fr\'echet derivative, i.e. there exists a constant $L_{\mathcal{F}}>0$  such that
   \(
        \|\mathcal{F}'(w)\| \leq L_{\mathcal{F}}
\) 
for all $w\in B_{2\rho}(\tilde{w})$.
    \item There exists a constant $\mu \in (0,1)$ such that
\(
    \| \mathcal{F}(w) - \mathcal{F}(\bar{w}) -\mathcal{F}'(\bar{w})(w-\bar{w}) \| \leq  \mu \|\mathcal{F}(w)-\mathcal{F}(\bar{w}) \|
\)
for all $w, \bar{w} \in B_{2\rho}(\tilde{w})$.
     \item $\hat{A}$ is a bounded linear operator with continuous Fr\'echet derivative. Moreover, there exists a constant $L_{\hat{A}} >0$ such that
\(
 \|\hat{A}'(w)\| \leq L_{\hat{A}}
\)
for all $w\in B_{2\rho}(\tilde{w})$.

     \item There exist constants $C_{\max}^{\hat{A}} \geq C_{\min}^{\hat{A}} >0$ such that
    \(
         C_{\min}^{\hat{A}} \leq \| \hat{A}(w^\dagger)-v \| \leq C_{\max}^{\hat{A}}.
    \)

     \item There exists a constant $R_{\mathcal{F}}>0$ such that 
   \(
        \|\mathcal{F}'(w)^* [v^{\delta}- \mathcal{F}(w)] \| \leq R_{\mathcal{F}}
   \)
   for all $w\in B_{2\rho}(\tilde{w})$ and $v^\delta \in \mathcal{V}$.
     
     \item For all $w, \bar{w} \in B_{2\rho}(\tilde{w})$ and $v^\delta \in \mathcal{V}$, there exist constants $\tilde{L}_\mathcal{F},\tilde{L}_{\hat{A}} >0$ such that 
     \begin{align*}
        \|\mathcal{F}'(w)^* [v^{\delta}- \mathcal{F}(w)]-\mathcal{F}'(\bar{w})^* [v^{\delta}- \mathcal{F}(\bar{w})]\| \leq &\, \tilde{L}_\mathcal{F} \|w-\bar{w}\|,
        \\
        \|\hat{A}' (w)^* [v^{\delta}-\hat{A}(w)]-\hat{A}'(\bar{w})^* [v^{\delta}- \hat{A}(\bar{w})]\| \leq &\, \tilde{L}_{\hat{A}} \|w-\bar{w}\|.
     \end{align*}
\end{enumerate}
\end{assumption}
\noindent Assumption \ref{Assum} provides the analytical framework required for the convergence analysis of \texttt{D-STAR} method \eqref{D-STAR}. Specifically, (A1) ensures the local boundedness of the Fr\'echet derivative, while (A2), which is known as the tangential cone condition, quantifies the nonlinearity of $\mathcal{F}$. These two conditions have been verified for nonlinear integral equations and parameter identification problems in PDEs, see~\cite{hanke1995convergence}. The condition (A3) imposes the local boundedness of the Fr\'echet derivative of data-driven operator $\hat{A}$, whereas (A4) ensures that $\hat{A}$ provides only a partial information for the exact data $v$. Both assumptions are standard in the analysis of data-driven regularization methods \cite{zhou2025data}. Condition (A5) is required to establish the well-posedness of \texttt{D-STAR} method~\eqref{D-STAR}, particularly due to the presence of the regularization weight $\lambda^\delta(t)$ in the drift term.
In addition, (A6) provides the conditions required for the well-posedness of \texttt{D-STAR} method~\eqref{D-STAR} due to presence of terms containing $\mathcal{F}$ and $\hat{A}$ in drift. It was originally introduced only for the operator $\mathcal{F}$ in \cite{long2025stochastic}.

\begin{remark}
    We emphasize that (A5) of Assumption \ref{Assum} can be interpreted as a local boundedness condition on the Fr\'echet derivative of the residual 
     $\|\mathcal{F}(w)-v^\delta\|^2/2$. Since $\mathcal{F}'(w)^*$ is a bounded linear operator, so (A5) is well-defined. Moreover, it can be verified for parameter identification problems in PDEs and nonlinear integral equations see~\cite{hanke1995convergence}.
\end{remark}

\begin{proposition}\label{R:lin;gr}
   By using (A3), (A4) and (A6) of Assumption \ref{Assum}, the following holds,
   \begin{align*}
        \|\hat{A}' (w)^* [v^{\delta}- \hat{A}(w)]\| \leq &\,  \tilde{L}_{\hat{A}} \big( \|w\| +  \|{w^\dagger}\| \big) + L_{\hat{A}} \big( \|v^{\delta}-v\| +  C_{\max}^{\hat{A}} \big) \quad \text{for all} \quad w  \in \mathcal{D}(\mathcal{F}) ~~\text{and }~~ v^{\delta} \in \mathcal{V}.
   \end{align*}
\end{proposition}

\begin{proof}
Using triangle inequality along with (A3), (A4) and (A6) of Assumption \ref{Assum}, one obtains
   \begin{align*}
        \|\hat{A}' (w)^* &\, [v^{\delta}- \hat{A}(w)]\| \leq  \|\hat{A}' (w)^* [v^{\delta}- \hat{A}(w)]-\hat{A}' (w^{\dagger})^* [v^{\delta}- \hat{A}(w^{\dagger})]\| + \|\hat{A}' (w^{\dagger})^* [v^{\delta}- \hat{A}(w^{\dagger})]\|
        \\
        \leq &\,  \tilde{L}_{\hat{A}}  \|w-{w^{\dagger}}\| + L_{\hat{A}} \|v^{\delta}- \hat{A}(w^{\dagger})\|
        \leq \tilde{L}_{\hat{A}}  \big( \|w\| + \|{w^{\dagger}}\| \big) + L_{\hat{A}} \big( \|v^{\delta}-v\| +  C_{\max}^{\hat{A}} \big). \qedhere
   \end{align*}
\end{proof}

\noindent
We now introduce the stopping rules for \texttt{D-STAR} method \eqref{D-STAR} to ensure appropriate termination in presence of the noisy data $v^\delta$ and to avoid both premature termination and overfitting of noise. For known noise level $\delta$, we employ a modified discrepancy principle, presented in Rule \ref{rule:MDP}. However, in many practical applications, noise level $\delta$ is not known and only the noisy data $v^\delta$ is available. In such cases, modified discrepancy principle cannot be applied. Therefore, we introduce a heuristic stopping criterion given in Rule \ref{rule:HR}.
Let us start by introducing the following modified discrepancy principle for known noise level, which is a variant of Morozov\textquotesingle{s} discrepancy principle \cite{morozov1966solution}.

\begin{rulenv}[Modified discrepancy principle] \label{rule:MDP}
Let the noise level $\delta$ is known and the noisy data $v^\delta$ satisfies \eqref{noise:lev}. Then, the stopping time $t^*:=t^*(\delta,v^\delta)$ is defined as
\begin{align*}
    t^* = \inf\{t>0: \hat{G}(t) \leq 0\},
\end{align*}
where 
\(
    \hat{G}(t):=  \mathbb E \big[\|v^{\delta}-\mathcal{F}(x^{\delta}(t)) \|^2 \big] - \tau^2 \delta^2,
\)
in which the constant $\tau>1$ is chosen such that
\begin{align*}
    \tau > \frac{2(\mu+1)}{2-(2+\epsilon)\mu}, \quad \epsilon \in \Big( 0, \frac{2}{\mu}-2 \Big) 
\end{align*}
which is essential for the proof of Lemma \ref{lem:mono}. 
\end{rulenv}

\noindent
Next, motivated by \cite{hanke1996general}, we introduce the following heuristic stopping criterion for the noisy data \(v^\delta\) in case of unknown noise level. A similar strategy for Banach space framework is studied in  \cite{jin2016hanke}.

\begin{rulenv}[Heuristic rule]\label{rule:HR}
    Let $a \geq 1$ be a fixed constant and $v^{\delta}$ denotes the noisy data. Assume
    \begin{align*}
        \alpha (t,v^{\delta}):= (t+a) \mathbb E \big[\|v^{\delta}-\mathcal{F}(x^{\delta}(t))\|^2 \big].
    \end{align*}
    Then, the stopping time $\tilde{t}^* := \tilde{t}^*(v^{\delta})$ is defined as
 \(
        \tilde{t}^* \in \arg\inf \{\alpha (t,v^{\delta}): t \in [0,t_\infty] \},
\)
    where $t_\infty$ is the largest number such that $x^{\delta}(t) \in \mathcal{D}(\mathcal{F})$ for all $t \in [0,t_\infty]$.
\end{rulenv}
\noindent Rule \ref{rule:HR} is straightforward to implement. During the iterative process, the values of $\alpha (t,v^{\delta})$ are recorded against $t$. After a sufficiently large number of iterations, method is terminated, and the stopping time $\tilde{t}^*$ is chosen as the minimizer of $\alpha (t,v^{\delta})$. 

While the proposed rule avoids the need for a prior knowledge of the noise level $\delta$, its theoretical justification is considerably more challenging. In particular, {Bakushinskii\textquotesingle{s}} veto \cite{bakushinskii1985remarks} states that heuristic rules cannot guarantee convergence in the worst-case setting. Consequently, the convergence analysis requires an additional assumption on the noisy data. Since \texttt{D-STAR} is a stochastic regularization method, the residual is itself a random quantity, therefore it is natural to formulate the \textit{noise condition} in stochastic sense. Hence, in the present work, we impose the following noise condition.

\begin{assumption}\label{as:HR} 
Let $v^\delta$ be the noisy data and   satisfies  $\|v^\delta -v\| \to 0$ as $\delta \to 0$. In addition, assume that there exists a constant $\kappa>0$ such that
    \begin{align*}
       \mathbb E \big[\|v^{\delta} - \mathcal{F}(w)\|^2 \big] \geq \kappa \|v^{\delta}-v\|^2
    \end{align*}
    for every $w\in \mathcal{S}(v^{\delta}):= \{x^{\delta}(t) : t \in [0,t_\infty] \}$, where $x^{\delta}(t)$ is the solution of \texttt{D-STAR} method \eqref{D-STAR}.
\end{assumption}
\noindent The above assumption ensures that the expected residual
$ \mathbb E \big[\|v^{\delta}- \mathcal{F}(w)\|^2 \big]$ does not become significantly smaller than $\|v^{\delta}-v\|^2$. Such  assumption in deterministic setting is standard for the analysis of heuristic parameter choice rules \cite{bajpai2026hanke,fu2020analysis, hanke1996general,jin2010heuristic,jin2017heuristic,kindermann2020convergence,real2024hanke} and it is essential for circumventing Bakushinskii's veto \cite{bakushinskii1985remarks}. This is an abstract condition which is generally difficult to verify in practice. In this direction, a detailed numerical study for nonlinear Landweber iteration is presented in \cite{hubmer2022numerical} and more recently, for an acoustic-electric tomography in \cite{liang2026complete}.


\section{Convergence Analysis of \texttt{D-STAR} Method under Modified Discrepancy Principle}\label{sec:MDP}
\noindent
This section illustrates the regularizing property of \texttt{D-STAR} method \eqref{D-STAR} under a modified discrepancy principle given in Rule \ref{rule:MDP} with a known noise level $\delta$ satisfying \eqref{noise:lev}. 
To this end, we first give the detailed description of \texttt{D-STAR} method \eqref{D-STAR} along with its monotonicity result and finite termination property in Subsection \ref{subsec:method}, then  its convergence to true solution of \eqref{IP} for the exact data $v$ in Subsection \ref{subsec:ex:conv}. Finally, Subsection \ref{subsec:noisy:conv}, establishes  its regularizing property for the noisy data.

\subsection{The method} \label{subsec:method}
Let us start by introducing the function  $g^{\delta}(t,x^\delta(t))$ as
   \begin{align}
       g^{\delta}(t,x^\delta(t)):=g^{\delta}(t) = \delta h(t), \label{def:f}
   \end{align}
 where $h(t)$ is a monotonically decreasing function such that
 \begin{align*}
     \lim_{t \to \infty} h(t)=0, \quad  h(0)= \tau \sqrt{\mu \epsilon}, \quad \epsilon \in \Big( 0, \frac{2}{\mu}-2 \Big).
 \end{align*}
Throughout this section, we adapt the definition of $g^{\delta}(t,x^\delta(t))$ as defined in \eqref{def:f} and for notational convenience, it is denoted as $g^\delta(t)$. Note that $g^\delta(t)$ depends on the noise level $\delta$ but independent of  $x^\delta(t)$ and chosen in such a way that it decays gradually with time. This choice of $g^{\delta}(t,x^\delta(t))$ is motivated by the analytical framework and the formulation of modified discrepancy principle.
Moreover, one can observe that
\begin{align}
    |g^{\delta}(t)|^2 = \delta^2 |h(t)|^2 \leq \delta^2 |h(0)|^2 \leq \delta^2 \tau^2 \mu  \epsilon < \mu \epsilon \mathbb E \big[\|v^{\delta}-\mathcal{F}(x^{\delta}(t)) \|^2 \big]\label{eq:MDP:f;del}
\end{align}
for all $t < t^*$. This estimate will be crucial for the subsequent analysis.

Next, we define the upper bound $\Gamma$ of the regularization weight $\lambda^{\delta}(t)$,  which is needed for the well-posedness of \texttt{D-STAR} method \eqref{D-STAR} with \eqref{def:f}. Specifically, $\Gamma$ is chosen as $\sup_{t\geq0} \lambda^\delta(t) = \Gamma :=  R_{\mathcal{F}}^2/\beta $
 so that \eqref{as:lamb,del} in Lemma \ref{lem:mono} is also satisfied, where constant $\beta>0$ depends on $C_{\max}^{\hat{A}},\delta,L_{\hat{A}},\tilde{w},$ and $ \rho$. Its explicit value is provided in Lemma \ref{lem:mono}. 
\begin{remark}
 Since $\mu$ quantifies the nonlinearity of $\mathcal{F}$, we introduce $\epsilon$ in the definition of $h(0)$ to ensure that $\mu \in(0,1)$. Indeed, in the nonlinear case, the restriction on $\tau$ in Rule \ref{rule:MDP} yields $\mu$ only in $(0,2/3)$ without this adjustment. When $\mathcal{F}$ is linear, i.e. $\mu =0$, no additional adjustment through $\epsilon$ is required in this case, and we set $h(0) = \tau$.
 \end{remark}

\noindent The pseudo-code of \texttt{D-STAR} method \eqref{D-STAR} with \eqref{def:f} under a modified discrepancy principle is presented in Algorithm \ref{alg:MDP}.

\begin{algorithm}
\caption{\texttt{D-STAR} with modified discrepancy principle.}\label{alg:MDP}
 \begin{description}[leftmargin=0pt, style=sameline]

\item[\textbf{Input}]
Forward operator $\mathcal{F}$, noisy data $v^\delta$, noise level $\delta$, parameters $\tau, \beta$, initial value $\tilde{w}$
and a function $h(t)$.

\item[\textbf{Training}]
Construct $\widehat{A}$ using training data
$(w_i,\mathcal{F}(w_i))_{1\leq i\leq m}$ such that
$\widehat{A}(w_i):=\mathcal{F}(w_i)$.

\item[\textbf{Repeat}]

\begin{enumerate}
        \item For $t \ge 0$, choose a  data-driven parameter $\lambda^\delta(t)\leq R_{\mathcal{F}}^2/\beta$ such that 
\begin{align*}
    \lambda^{\delta}(t)  \leq {K_{\lambda^{\delta}} \mathbb E \big[\|\mathcal{F}' (x^{\delta}(t))^* [v^{\delta}-\mathcal{F} (x^{\delta}(t))]\|^2 \big]}/\beta \quad  \text{and} \quad 
     2- (2+\epsilon)\mu  -\frac{2(\mu +1)}{\tau}  -K_{\lambda^{\delta}} L_{\mathcal{F}}^2 > 0.
\end{align*}

     \item Solve:     
\(
    dx^{\delta}(t) = \big(\mathcal{F}' (x^{\delta}(t))^* [v^{\delta} - \mathcal{F}(x^{\delta}(t)) ] + \lambda^{\delta}(t) \hat{A}'(x^{\delta}(t))^* [v^{\delta}-\hat{A}(x^{\delta}(t))] \big)  dt + \delta h(t) dB(t).
\)
    \end{enumerate}

\item[\textbf{Until}] $\mathbb E \big[\|v^\delta-\mathcal{F}(x^\delta(t))\|^2 \big] \leq \tau^2 \delta^2$. 

\item[\textbf{Output}]  $x^\delta(t^*)$.
\end{description}
\end{algorithm}

\noindent Since \texttt{D-STAR} method \eqref{D-STAR} is formulated as  a stochastic differential equation, it is essential to establish its well-posedness. Accordingly, we first establish the following lemma.
\begin{lemma}\label{lem:well}
    Let Assumption \ref{Assum} be satisfied for all $w, \bar{w} \in \mathcal{W}$. Then for every $t \in [0,T]$, $v^{\delta} \in \mathcal{V}$ and $g^\delta(t) \in L^\infty(\mathbb{R}_+)$, the proposed \texttt{D-STAR} method~\eqref{D-STAR} admits a unique mild solution $x^{\delta}(t)$.
\end{lemma}

\begin{proof}
    We proceed by verifying that \texttt{D-STAR} method \eqref{D-STAR} satisfies the assumptions of Theorem \ref{th:pre:well}, which  concludes the existence of its unique mild solution.
  Let us define
    \begin{align*}
    H=0, \quad  F_1(t,x^\delta)=   \mathcal{F}'(x^\delta)^* [v^{\delta} - \mathcal{F}(x^\delta) ] + \lambda^{\delta}(t) \hat{A}'(x^\delta)^* [v^{\delta}-\hat{A}(x^\delta)],
     \quad   \text{and} \quad
        F_2(t,x^\delta)=  g^{\delta}(t).
    \end{align*}
    Then, clearly $H=0$ is a generator of a $C_0$-semigroup and from (A1), (A3) of Assumption \ref{Assum} along with \eqref{def:f}, $F_1$ and $F_2$ are measurable maps .
    Furthermore, condition \eqref{eq:pre:lip} follows directly from (A6) of Assumption \ref{Assum} and for verification of condition \eqref{eq:pre:gr}, by using Proposition \ref{R:lin;gr} and (A5) of Assumption \ref{Assum} along with \eqref{noise:lev}, one obtains
    \begin{align*}
        \|F_1(t, x^\delta)&\,   \|^2 \leq   2 \|\mathcal{F}'(x^\delta)^* [v^{\delta} - \mathcal{F}(x^\delta) ] \|^2 + 2\lambda^{\delta}(t)^2 \|\hat{A}'(x^\delta)^* [v^{\delta}-\hat{A}(x^\delta)] \|^2
        \\
         \leq &\, 2{R}_\mathcal{F}^2 
          + 2\Gamma^2 \big(\tilde{L}_{\hat{A}} (  \|x^\delta\| +  \|{w^\dagger}\|) + L_{\hat{A}}( \|v^\delta-v \| +  C_{\max}^{\hat{A}}) \big)^2
          \\
          \leq &\, \big( 2{R}_\mathcal{F}^2 + 8 \Gamma^2 \tilde{L}_{\hat{A}}^2 \|{w^\dagger}\|^2 + 4 \Gamma^2L_{\hat{A}}^2(\|v^\delta-v \| +  C_{\max}^{\hat{A}})^2 \big) + 8 \Gamma^2 \tilde{L}_{\hat{A}}^2\|x^\delta\|^2
          \leq \bar{K}_1 (1+\|x^\delta\|^2)
    \end{align*}
where $\bar{K}_1 = \max\{2{R}_\mathcal{F}^2 + 8 \Gamma^2 \tilde{L}_{\hat{A}}^2 \|{w^\dagger}\|^2 + 4\Gamma^2L_{\hat{A}}^2(\delta  +  C_{\max}^{\hat{A}})^2, 8 \Gamma^2 \tilde{L}_{\hat{A}}^2\}$. Thus, by using  \eqref{eq:MDP:f;del}, $F_1(t,x^\delta)$ and $F_2(t,x^\delta)$ satisfies condition \eqref{eq:pre:gr}. Therefore, all the assumptions of Theorem \ref{th:pre:well} are fulfilled, which implies that \texttt{D-STAR} method \eqref{D-STAR} admits a unique mild solution and have finite moments satisfying \eqref{eq:pre:MB1}. Moreover, the solution satisfies
\begin{align}
    x^{\delta}(t) = &\, \tilde{w} + \int_0^t \big(\mathcal{F}' (x^{\delta}(s))^* [v^{\delta} - \mathcal{F}(x^{\delta}(s)) ] + \lambda^{\delta}(s) \hat{A}'(x^{\delta}(s))^* [v^{\delta}-\hat{A}(x^{\delta}(s))] \big)  ds + \int_0^t g^{\delta}(s) dB(s).  \label{SAR:sol} 
\end{align}
\end{proof}
\noindent
The following lemma establish that the map $t \mapsto \mathbb E\big[ \|x^{\delta}(t)- w^\dagger \|^2 \big]$ is monotonically non-increasing for $t < t^*$. This property is essential for proving the existence of finite stopping time $t^*$ associated with Rule \ref{rule:MDP}.

\begin{lemma}\label{lem:mono}
Suppose Assumption \ref{Assum} holds and $x^{\delta}(t)$ denotes the solution of \texttt{D-STAR} method~\eqref{D-STAR}. Let $t^*$ be the stopping time associated with Rule \ref{rule:MDP} and there exists a constant $K_{\lambda^{\delta}}$ such that
\begin{align}
    \lambda^{\delta}(t)  & \leq {K_{\lambda^{\delta}} \mathbb E \big[\|\mathcal{F}' (x^{\delta}(t))^* [v^{\delta}-\mathcal{F} (x^{\delta}(t))]\|^2 \big]}/\beta  \quad \text{and} \quad \mathcal{H}:= 2-  (2+\epsilon)\mu -\frac{2(\mu +1)}{\tau}-K_{\lambda^{\delta}} L_{\mathcal{F}}^2 > 0,  \label{as:lamb,del}
\end{align}
where $\beta = \delta^2+ (C_{\max}^{\hat{A}})^2 +3L_{\hat{A}}^2 \rho^2/2 $. 
Then, for every $t< t^*$, the following is true,
    \begin{align}
        \frac{d}{dt} \mathbb E \big[\|x^{\delta}(t)- w^\dagger \|^2 \big] \leq -  \mathcal{H} \mathbb E \big[\|v^{\delta}-\mathcal{F} (x^{\delta}(t))\|^2 \big] \quad \text{and} \quad x^{\delta}(t) \in B_{2\rho}(\tilde{w}). \label{eq:di x^n,x*}
    \end{align}
\end{lemma}

\begin{proof}
   From \eqref{SAR:sol}, one can see that
   \begin{align*}
       x^{\delta}(t)- w^\dagger & = \tilde{w} -w^\dagger + \int_0^t \big(\mathcal{F}' (x^{\delta}(s))^* [v^{\delta} - \mathcal{F}(x^{\delta}(s)) ] + \lambda^{\delta}(s) \hat{A}'(x^{\delta}(s))^* [v^{\delta}-\hat{A}(x^{\delta}(s))] \big)ds
       + \int_0^t g^{\delta}(s) dB(s).
   \end{align*}
   Now, by applying It\^o\textquotesingle{s} formula given in Theorem \ref{th:pre;ito}, one obtains
   \begin{align*}
      & \|x^{\delta}(t)  - w^\dagger\|^2 =  \| \tilde{w} -w^\dagger\|^2 +  2 \int_0^t \big \langle x^{\delta}(s)- w^\dagger,\mathcal{F}' (x^{\delta}(s))^* [v^{\delta} - \mathcal{F}(x^{\delta}(s)) ] \big \rangle ds
       \\
       & \, + 2 \int_0^t \big \langle x^{\delta}(s)- w^\dagger, \lambda^{\delta}(s) \hat{A}'(x^{\delta}(s))^* [v^{\delta}-\hat{A}(x^{\delta}(s))] \big \rangle 
       ds   + 2 \int_0^t \big \langle x^{\delta}(s)- w^\dagger, g^{\delta}(s) dB(s) \big \rangle  + \int_0^t | g^{\delta}(s)|^2 ds
   \end{align*}
   which on taking expectation and applying the martingale property of stochastic integral given in Theorem \ref{th:pre;stoc-int} yields
   \begin{align}
       \mathbb E \big[\|x^{\delta}(t)- &\, w^\dagger \|^2 \big] =  \mathbb E \big[\| \tilde{w} -w^\dagger\|^2 \big] +  2\mathbb E\int_0^t \big[ \big \langle x^{\delta}(s)- w^\dagger,\mathcal{F}' (x^{\delta}(s))^* [v^{\delta} - \mathcal{F}(x^{\delta}(s)) ] \big \rangle \big] ds \notag 
       \\
       & +  2\mathbb E\int_0^t \big[ \big \langle x^{\delta}(s)- w^\dagger, \lambda^{\delta}(s) \hat{A}'(x^{\delta}(s))^* [v^{\delta}-\hat{A}(x^{\delta}(s))] \big \rangle \big] ds  +  \int_0^t |g^{\delta}(s)|^2 ds. \label{eq:Ito:di x^n,x*}
   \end{align}
     By using  fundamental theorem of calculus,
    \begin{align}
        \frac{d}{dt} &\, \mathbb E \big[\|  x^{\delta}(t)- w^\dagger \|^2 \big] =  2 \mathbb E \big[ \big \langle x^{\delta}(t)- w^\dagger,\mathcal{F}' (x^{\delta}(t))^* [v^{\delta} - \mathcal{F}(x^{\delta}(t)) ] \big \rangle \big] \notag
        \\
        &   + 2 \lambda^{\delta}(t) \mathbb E \big[ \big \langle x^{\delta}(t)- w^\dagger,  \hat{A}'(x^{\delta}(t))^* [v^{\delta}-\hat{A}(x^{\delta}(t))] \big \rangle \big] +  |g^{\delta}(t)|^2 
        :=  \phi_{\mathcal{F}} + \phi_{\hat{A}} +  |g^{\delta}(t)|^2, \label{eq:mon:der:dif}
    \end{align}
     where $\phi_{\mathcal{F}}$ and $ \phi_{\hat{A}}$ denote the terms associated with $\mathcal{F}$ and $\hat{A}$, respectively. Thus, for the estimation of $\phi_{\mathcal{F}}$, notice that
\begin{align*}
     \phi_{\mathcal{F}} :=&\, 2 \mathbb E \big[ \big \langle x^{\delta}(t)- w^\dagger,\mathcal{F}' (x^{\delta}(t))^* [v^{\delta} - \mathcal{F}(x^{\delta}(t)) ] \big \rangle \big] =
      2\mathbb E \big[ \big \langle\mathcal{F}' (x^{\delta}(t)) (x^{\delta}(t)- w^\dagger), v^{\delta} - \mathcal{F}(x^{\delta}(t)) \big \rangle \big]
    \\
     = &\, 2\mathbb E \big[ \big \langle \mathcal{F}(w^\dagger) - \mathcal{F}(x^{\delta}(t)) -\mathcal{F}' (x^{\delta}(t)) (w^\dagger-x^{\delta}(t)), v^{\delta} - \mathcal{F}(x^{\delta}(t)) \big \rangle \big] 
    \\
    & + 2\mathbb E \big[ \big \langle \mathcal{F}(x^{\delta}(t))-\mathcal{F}(w^\dagger),v^{\delta} - \mathcal{F}(x^{\delta}(t)) \big \rangle \big]
\end{align*}
which on using (A2) of Assumption \ref{Assum} and Cauchy-Schwarz inequality  yields
\begin{align}
    \phi_{\mathcal{F}} \leq &\, 2\mu \mathbb E \big[\|\mathcal{F}(w^\dagger) - \mathcal{F}(x^{\delta}(t))\| \|v^{\delta} - \mathcal{F}(x^{\delta}(t)) \| \big] + 2\mathbb E \big[ \big \langle v^{\delta} -v  ,v^{\delta} - \mathcal{F}(x^{\delta}(t)) \big \rangle \big]  - 2 \mathbb E \big[\|v^{\delta} - \mathcal{F}(x^{\delta}(t)) \|^2 \big] \notag
    \\
     \leq &\, 2 \mu \mathbb E  \Big [\|v^{\delta} - \mathcal{F}(x^{\delta}(t)) \| \big[ \|v^{\delta} - \mathcal{F}(x^{\delta}(t))\| + \|v- v^{\delta}\| \big] \Big] + 2 \mathbb E \big[ \| v^{\delta} -v\|  \|v^{\delta} - \mathcal{F}(x^{\delta}(t))\| \big] \notag
     \\
     & - 2\mathbb E \big[\|v^{\delta} - \mathcal{F}(x^{\delta}(t)) \|^2 \big] \notag
     \\
      \leq &\, 2(\mu-1) \mathbb E \big[ \|v^{\delta} - \mathcal{F}(x^{\delta}(t)) \|^2 \big] + 2(\mu+1) \mathbb E \big[ \|v^\delta -v\| \|v^{\delta} - \mathcal{F}(x^{\delta}(t)) \| \big] . \label{eq:phi-F}
\end{align}
Furthermore, for the estimation of $\phi_{\hat{A}}$, using the linear property of $\hat{A}$ and Jensen\textquotesingle{s} inequality, one can observe that
\begin{align}
    \phi_{\hat{A}} := &\, 2 \lambda^{\delta}(t) \mathbb E \big[ \big \langle x^{\delta}(t)- w^\dagger,  \hat{A}'(x^{\delta}(t))^* [v^{\delta}-\hat{A}(x^{\delta}(t))] \big \rangle \big] \notag
    \\
    = &\, 2 \lambda^{\delta}(t) \mathbb E \big[ \big \langle \hat{A}(x^{\delta}(t)) - \hat{A}(w^\dagger), v^{\delta}-\hat{A}(x^{\delta}(t)) \big \rangle \big] \notag
    \\
    \leq &\,  2 \lambda^{\delta}(t) \Big[ -\mathbb E \big[ \|v^{\delta}-\hat{A}(x^{\delta}(t)) \|^2 \big] + \mathbb E \big[\|v^{\delta}-\hat{A}(w^\dagger)\| \|v^{\delta}-\hat{A}(x^{\delta}(t))\| \big] \Big] \notag
    \\
    \leq &\, 2 \lambda^{\delta}(t)  \Big[ -\mathbb E \big[\|v^{\delta}-\hat{A}(x^{\delta}(t)) \|^2 \big] + \|v^{\delta}-\hat{A}(w^\dagger)\| \, \big(\mathbb E \big[  \|v^{\delta}-\hat{A}(x^{\delta}(t)) \|^2 \big] \big)^{1/2} \Big].\label{eq:phi-A*}
\end{align}
By elementary inequality $-r^2+ar = -(r^2-ar+ a^2/4)+ a^2/4 \leq a^2/4$ for all  $a,r\geq 0$, along with \eqref{noise:lev} and (A4) of Assumption \ref{Assum},
\begin{align}
    \phi_{\hat{A}} 
     \leq &\, 2 \lambda^{\delta}(t) \frac{1}{4} \|v^{\delta}-\hat{A}(w^\dagger)\|^2 \leq  \lambda^{\delta}(t) \big[\|v^\delta-v\|^2+ \|v-\hat{A}(w^\dagger)\|^2  \big] \leq \lambda^{\delta}(t)\big(\delta^2  +(C_{\max}^{\hat{A}})^2 \big) \label{eq:phi-A-int}
\end{align}
which on using \eqref{as:lamb,del} and (A1) of Assumption \ref{Assum}, yields
\begin{align}
   \phi_{\hat{A}} \leq &\,  K_{\lambda^{\delta}} L_{\mathcal{F}}^2 \mathbb E \big[ \|v^{\delta}-\mathcal{F} (x^{\delta}(t))\|^2 \big]. \label{eq:phi-A} 
\end{align}
Therefore, on substituting \eqref{eq:phi-F} and \eqref{eq:phi-A} into \eqref{eq:mon:der:dif}, then using \eqref{noise:lev}, \eqref{eq:MDP:f;del} and Jensen\textquotesingle{s} inequality, one obtains
\begin{align}
     \frac{d}{dt} \mathbb E &\, \big[\|x^{\delta}(t)- w^\dagger \|^2 \big] \leq  2(\mu-1) \mathbb E \big[ \|v^{\delta} - \mathcal{F}(x^{\delta}(t)) \|^2 \big] + 2 \delta (\mu+1) \mathbb E \big[\|v^{\delta} - \mathcal{F}(x^{\delta}(t)) \| \big]  \notag
     \\
     &  + K_{\lambda^{\delta}} L_{\mathcal{F}}^2 \mathbb E \big[ \|v^{\delta}-\mathcal{F} (x^{\delta}(t))\|^2 \big] +  |g^{\delta}(t)|^2 \notag
     \\
    \leq &\, \big( (2+\epsilon)\mu-2 +  K_{\lambda^{\delta}} L_{\mathcal{F}}^2 \big) \mathbb E \big[ \|v^{\delta} - \mathcal{F}(x^{\delta}(t)) \|^2 \big] + {2\delta (\mu+1)} \big(\mathbb E  \big[\|v^{\delta} - \mathcal{F}(x^{\delta}(t)) \|^2 \big ] \big)^{1/2} . \label{eq:der,dif,xn,exa}
\end{align}
Thus, 
\begin{align*}
    \frac{d}{dt} \mathbb E \big[\| &\, x^{\delta}(t)  - w^\dagger \|^2 \big] \leq   \big( (2+\epsilon)\mu-2 +  K_{\lambda^{\delta}}L_{\mathcal{F}}^2 \big) \mathbb E \big[ \|v^{\delta} - \mathcal{F}(x^{\delta}(t)) \|^2 \big] + \frac{2(\mu+1)}{\tau} \mathbb E \big[ \|v^{\delta} - \mathcal{F}(x^{\delta}(t)) \|^2 \big]
    \\
    = &\, - \Big( 2-(2+\epsilon)\mu -  K_{\lambda^{\delta}} L_{\mathcal{F}}^2 -\frac{2(\mu +1)}{\tau} \Big) \mathbb E \big[\|v^{\delta} - \mathcal{F}(x^{\delta}(t)) \|^2 \big] 
    =  - \mathcal{H}\mathbb E \big[\|v^{\delta} - \mathcal{F}(x^{\delta}(t)) \|^2 \big] \qedhere
\end{align*}
which implies that $\mathbb E \big[\| x^{\delta}(t)  - w^\dagger \|^2 \leq \| \tilde{w} - w^\dagger \|^2 \leq \rho^2$ as $\mathcal{H}>0$. Thus, $x^{\delta}(t) \in B_{\rho}(w^\dagger)$ for all $t< t^*$. Therefore one obtains $\mathbb E \big[\|w^\delta(t) -\tilde{w} \|^2\big] \leq 2\mathbb E \big[\|w^\delta(t)-w^\dagger\|^2 \big] + 2\|w^\dagger-\tilde{w}\|^2 \leq 4 \rho^2$, which implies that $x^{\delta}(t) \in B_{2\rho}(\tilde{w})$ for all $t< t^*$. This completes the proof.
\end{proof}

\noindent
Next, we prove the existence of a finite stopping time $t^*$ associated with Rule \ref{rule:MDP}.
\begin{theorem}\label{th:uni,t}
  Let assumptions of Lemma \ref{lem:mono} be satisfied and $x^{\delta}(t)$ denotes the solution of \texttt{D-STAR} method~\eqref{D-STAR}. Then, there exists a finite $t^*$ such that
    \begin{align*}
        \mathbb E \big[\|v^{\delta}-\mathcal{F}(x^{\delta}(t^*))\|^2 \big] \leq \tau^2 \delta^2  <  \mathbb E \big[\|v^{\delta}-\mathcal{F}(x^{\delta}(t)) \|^2 \big] \quad \text{for all}~ t < t^*.
    \end{align*}
\end{theorem}

\begin{proof}
   Suppose there does not exists any finite $t^*$, then for all $t \geq 0$, one gets
 \(
       \mathbb E \big[\|v^{\delta} -\mathcal{F}(x^{\delta}(t))\|^2 \big] > \tau^2 \delta^2. 
\)
   Now, from Lemma \ref{lem:mono},
   \begin{align*}
        \frac{d}{dt} \mathbb E \big[\|x^{\delta}(t)- w^\dagger \|^2 \big] \leq - \Big( 2-(2+\epsilon)\mu  - K_{\lambda^{\delta}} L_{\mathcal{F}}^2 -\frac{2(\mu +1)}{\tau} \Big) \mathbb E \big[\|v^{\delta}-\mathcal{F} (x^{\delta}(t))\|^2 \big]
    \end{align*}
   which on integration gives
    \begin{align*}
        \int_0^\infty \mathbb E \big[\|v^{\delta}-\mathcal{F} (x^{\delta}(t))\|^2 \big] dt \leq &\, -\frac{\tau}{\tau \big(2-(2+\epsilon)\mu  - K_{\lambda^{\delta}} L_{\mathcal{F}}^2 \big) - 2(\mu +1) } \int_0^\infty \frac{d}{dt} \mathbb E \big[ \|x^{\delta}(t)- w^\dagger \|^2 \big] dt.
    \end{align*}
   By fundamental theorem of calculus, 
    \begin{align*}
       \infty =\int_0^\infty \tau^2 \delta^2 dt < \int_0^\infty \mathbb E \big[\|v^{\delta}-\mathcal{F} (x^{\delta}(t))\|^2 \big] dt \leq \frac{\tau}{\tau \big(2-(2+\epsilon)\mu  - K_{\lambda^{\delta}} L_{\mathcal{F}}^2 \big) - 2(\mu +1) }  \|\tilde{w}-w^\dagger\|^2 < \infty
    \end{align*}
    which leads to a contradiction. Hence, the existence of  finite $t^*$ is guaranteed.
\end{proof}


\subsection{Convergence for exact data} \label{subsec:ex:conv}
This subsection demonstrates the convergence of \texttt{D-STAR} method~\eqref{D-STAR} for exact data $v$. In this setting, the stochastic term vanishes and method reduces to its deterministic counterpart. Consequently, the expected squared \(L^2\)-norm coincides with the corresponding deterministic squared \(L^2\)-norm. The resulting pseudo-code is presented in Algorithm~\ref{alg:ex_data}. 

\begin{algorithm}
\caption{\texttt{D-STAR} with exact data. }\label{alg:ex_data}
 \begin{description}[leftmargin=*, labelindent=0pt, style=sameline]
\item[\textbf{Input}]
Forward operator $\mathcal{F}$, exact data $v$, parmeter $\gamma$ and initial value $\tilde{w}$.

\item[\textbf{Training}]
Construct $\widehat{A}$ using training data
$(w_i,\mathcal{F}(w_i))_{1\leq i\leq m}$ such that
$\widehat{A}(w_i):= \mathcal{F}(w_i)$.

\item[\textbf{Repeat}]

\begin{enumerate}
        \item For $t \geq 0$, choose a data-driven parameter $\lambda(t) \leq R_{\mathcal{F}}^2/\gamma$ such that 
\begin{align*}
    \lambda(t)  \leq {K_{\lambda} \mathbb E \big[\|\mathcal{F}' (w(t))^* [v-\mathcal{F} (w(t))]\|^2 \big]}/\gamma \quad \text{and} \quad \mathcal{H}_1:=
    2(1- \mu) - K_{\lambda} L_{\mathcal{F}}^2 >0.
\end{align*}
     \item Solve:    
\(
    dw(t) = \big(\mathcal{F}' (w(t))^* [v - \mathcal{F}(w(t)) ] + \lambda(t) \hat{A}'(w(t))^* [v -\hat{A}(w(t))] \big)  dt.
\)
    \end{enumerate}
\end{description}
\end{algorithm}

\noindent
Now, we first establish the following lemma.
\begin{lemma}\label{lem:exa:mono}
  Let Assumption \ref{Assum} be satisfied and $w(t)$ denotes the solution of \texttt{D-STAR} method \eqref{D-STAR} with exact data $v$. Assume there exists a constant $K_{\lambda}$
 such that
\vspace{0.1cm}
\begin{align}\label{eq:lam:ex}
    \lambda(t) \leq {K_{\lambda}  \mathbb E  \big[\|\mathcal{F}' (w(t))^* [v-\mathcal{F} (w(t))]\|^2 \big]}/ \gamma \quad \text{and} \quad 
    \mathcal{H}_1 = 2(1 - \mu) - K_{\lambda}L_{\mathcal{F}}^2  >0, 
\end{align}
where $\gamma=  (C_{\max}^{\hat{A}})^2 +3L_{\hat{A}}^2 \rho^2 /2  $.
Then for every $t\geq 0$, 
    \begin{align}\label{eq:di:ex,x^n,x*}
        \frac{d}{dt} \mathbb E \big[\|w(t)- w^\dagger \|^2 \big] \leq  -     \mathcal{H}_1\mathbb E \big[ \|v-\mathcal{F} (w(t))\|^2 \big] \quad \text{and} \quad
        \int_0^\infty\mathbb E \big[ \|v  -\mathcal{F} (w(t))\|^2 \big]  dt  < \infty. 
    \end{align}
\end{lemma}
\begin{proof}
    The proof follows by the arguments analogous to Lemma \ref{lem:mono} and fundamental theorem of calculus.
\end{proof}
\noindent
Next, we derive the following convergence theorem for the exact data case.

\begin{theorem}\label{th:exa,conv}
Suppose assumptions of Lemma \ref{lem:exa:mono} holds and $w(t)$ denotes the solution of \texttt{D-STAR} method \eqref{D-STAR} with the exact data $v$. Further, assume that $w^\dagger$ denotes a solution of \eqref{IP}. Then, 
    \begin{align*}
       \lim_{t \to \infty} \mathbb E \big[ \|w(t)-w^\dagger\|^2 \big] =0.
    \end{align*}
\end{theorem}

\begin{proof}
Let us define the error $\hat{e}(t):=w(t)-w^\dagger$ and we aim to show that $\{\hat{e}(t)\}_{t \geq
0}$ is a Cauchy sequence. Then, for an arbitrary $r>t$, one can observe that
 \begin{align}
     \mathbb E \big[\|\hat{e}(t)-\hat{e}(r)\|^2 \big] = \mathbb E \big[\|\hat{e}(t)\|^2 \big] - \mathbb E \big[\|\hat{e}(r)\|^2 \big] + 2 \mathbb E \big[ \langle \hat{e}(r)-\hat{e}(t),\hat{e}(r) \rangle \big]. \label{eq:ex:di:err}
 \end{align}
By Lemma~\ref{lem:exa:mono}, both $\mathbb E\big[\|\hat{e}(t)\|^2\big]$ and $\mathbb E\big[\|\hat{e}(r)\|^2\big]$ are monotonically decreasing and therefore converge to some $\vartheta > 0$. Consequently, one needs to analyze only the last term on the right-hand side of \eqref{eq:ex:di:err}. To this end, \eqref{D-STAR} with the exact data $v$ yields
   \begin{align}
      \mathbb E \big[ \langle \hat{e}(r)-\hat{e}(t),\hat{e}(r) \rangle \big] = &\,  \mathbb E \Big[ \Big \langle \int_t^r \big(\mathcal{F}' (w(s))^* [v- \mathcal{F}(w(s)) ] + \lambda(s) \hat{A}'(w(s))^* [v-\hat{A}(w(s))] \big) ds,\hat{e}(r) \Big \rangle \Big] \notag
      \\
     = &\,  \mathbb E \int_t^r \big[ \langle v- \mathcal{F}(w(s)),\mathcal{F}' (w(s))(w(r)-w^\dagger) \rangle \big] ds \notag
     \\
     & + \mathbb E \int_t^r \big[ \lambda(s) \langle  v-\hat{A}(w(s)), \hat{A}'(w(s))(w(r)-w^\dagger) \rangle \big] ds
     :=  \psi_{\mathcal{F}} + \psi_{\hat{A}}, \label{eq:psi-F+A}
   \end{align}
   where $\psi_{\mathcal{F}}$ and $\psi_{\hat{A}}$ represents the term associated with $\mathcal{F}$ and $\hat{A}$, respectively. 
   For the estimation of $\psi_{\mathcal{F}}$, using H\"older\textquotesingle{s}  inequality and (A2) of Assumption \ref{Assum}, one obtains
   \begin{align}
       \psi_{\mathcal{F}}:= &\, \mathbb E \int_t^r \big[ \langle v- \mathcal{F}(w(s)),\mathcal{F}' (w(s))(w(r)-w^\dagger) \rangle \big] ds \notag
       \\
       \leq &\,  \int_t^r \mathbb E \Big[\|v- \mathcal{F}(w(s))\| \left[ \|\mathcal{F}' (w(s))(w(s)-w^\dagger)\| +  \|\mathcal{F}' (w(s))(w(r)-w(s))\| \right] \Big] ds \notag
       \\
        \leq &\, \int_t^r \big( \mathbb E \big[  \|v-\mathcal{F}(w(s))\|^2 \big]\big)^{1/2} \big(  \mathbb E \big[  \|\mathcal{F}' (w(s))(w(s)-w^\dagger)\|^2  \big] \big)^{1/2}  ds \notag
        \\
        & + \int_t^r  \big(\mathbb E \big[  \|v-\mathcal{F}(w(s))\|^2 \big]\big)^{1/2}  \big( \mathbb E \big[ \|\mathcal{F}' (w(s))(w(r)-w(s))\|^2  \big] \big)^{1/2}  ds \notag
        \\
        \leq &\, 3(\mu+1) \int_t^r \mathbb E \big[ \|v-\mathcal{F}(w(s))\|^2 \big] ds. \label{eq:psi-F}
   \end{align}
   For the estimation of $\psi_{\hat{A}}$, by using (A3) of Assumption \ref{Assum} and H\"older\textquotesingle{s} inequality, one gets
   \begin{align*}
       \psi_{\hat{A}} := &\, \mathbb E  \int_t^r \big[ \lambda(s) \langle  v-\hat{A}(w(s)), \hat{A}'(w(s)) (w(r)-w^\dagger) \rangle \big] ds
       \\
       \leq &\, L_{\hat{A}} \int_t^r \lambda(s) \mathbb E \big [\|v-\hat{A}(w(s))\| \|w(r)- w^\dagger\| \big] ds
       \leq L_{\hat{A}} \int_t^r \lambda(s)  \big(\mathbb E \big[\|v-\hat{A}(w(s))\|^2 \big] \big)^{1/2} \big( \mathbb E \big[\|w(r)- w^\dagger\|^2 \big] \big) ^{1/2}  ds.
   \end{align*}
  By monotonically non-increasing property of $\mathbb E \left[ \|w(t)- w^\dagger \|^2 \right]$ as given in Lemma \ref{lem:exa:mono}, 
   \begin{align}
       \mathbb E \big[ \|w(t)- w^\dagger \|^2 \big] \leq  \|\tilde{w}-w^\dagger\|^2 \quad \text{for all } t\geq 0. \label{eq:dif:ex sc,tr}
   \end{align}
Moreover, (A3) and (A4) of Assumption \ref{Assum} together with \eqref{eq:dif:ex sc,tr}, yields
   \begin{align}
       \mathbb E\big[ \|v - \hat{A} &\, (w(t))\|^2  \big] \leq  2\mathbb E \big[ \|v-\hat{A}(w^\dagger)\|^2 \big] + 2\mathbb E \big[ \|\hat{A}(w^\dagger)-\hat{A}(w(t))\|^2 \big] \notag
       \\
       \leq &\, 2(C_{\max}^{\hat{A}})^2 + 2L_{\hat{A}}^2 \mathbb E \big[ \|w^\dagger-w(t)\|^2 \big] 
       \leq  2(C_{\max}^{\hat{A}})^2 + 2L_{\hat{A}}^2\|w^\dagger-\tilde{w}\|^2. \label{eq:dif:ex y,A}
   \end{align}
 On utilizing \eqref{eq:dif:ex sc,tr} and \eqref{eq:dif:ex y,A}, Young\textquotesingle{s} inequality and (A1) of Assumption \ref{Assum} along with \eqref{eq:lam:ex}, one has
  \begin{align}
       \psi_{\hat{A}}  \leq &\, L_{\hat{A}} \mathbb E \int_t^r \Big[ \lambda(s) \big( 2(C_{\max}^{\hat{A}})^2  + 2L_{\hat{A}}^2\|w^\dagger-\tilde{w}\|^2 \big)^{1/2} \|w^\dagger-\tilde{w}\| \Big] ds \notag
        \\
        \leq &\, \mathbb E \int_t^r \Big[ \lambda(s) \Big( (C_{\max}^{\hat{A}})^2  + \frac{3}{2}L_{\hat{A}}^2 \|w^\dagger-\tilde{w}\|^2 \Big) \Big]ds 
        \leq K_{\lambda} L_{\mathcal{F}}^2  \int_t^r \mathbb E \big[ \|v-\mathcal{F}(w(s))\|^2 \big] ds. \label{eq:psi-A}
  \end{align}
  Therefore, on using \eqref{eq:psi-F} and \eqref{eq:psi-A} in \eqref{eq:psi-F+A}, one obtains
  \begin{align*}
      \mathbb E \big[ \langle \hat{e}(r)-\hat{e}(t),\hat{e}(r) \rangle \big] \leq &\, \big( 3(\mu+1)+ K_{\lambda} L_{\mathcal{F}}^2  \big)  \int_t^r \mathbb E \big[ \|v-\mathcal{F}(w(s))\|^2 \big] ds
  \end{align*}
  which on using Lemma \ref{lem:exa:mono} yields
\(
      \lim_{t \to \infty} \mathbb E \big[ \langle \hat{e}(r)-\hat{e}(t),\hat{e}(r) \rangle \big] =0.
\)
  Thus, from \eqref{eq:ex:di:err}, one concludes that $\{\hat{e}(t)\}_{t\geq 0}$ is a Cauchy sequence. Hence, $\{w(t)\}_{t\geq 0}$ is also a Cauchy sequence and converges to $w^\dagger$, which completes the proof.
\end{proof}


\subsection{Convergence for noisy data}\label{subsec:noisy:conv}
This subsection establishes the main result of this section, which states the regularizing property of \texttt{D-STAR} method~\eqref{D-STAR} for the noisy data $v^\delta$ satisfying \eqref{noise:lev}. The proof relies on the following stability lemma.
\begin{lemma}\label{lem:stab}
Suppose Assumption \ref{Assum} holds and $\lambda^{\delta}(t) \to \lambda(t)$ as $\delta \to 0$ for any fixed $t$. Let $x^\delta(t)$ be the solution of \texttt{D-STAR} method \eqref{D-STAR}  and $w(t)$ be its solution with the exact data $v$. Then, for every $t\geq0$, 
    \begin{align*}
        \lim_{\delta \to 0 } \mathbb E \big [\|x^{\delta}(t)-w(t)\|^2 \big] =0.
    \end{align*}
\end{lemma}

\begin{proof}
Let us denote
\begin{align}
     \mathrm{F}^\delta(s) :=\mathcal{F}' (x^{\delta}(s))^* [v^{\delta} - \mathcal{F}(x^{\delta}(s)) ], & \quad \mathrm{F}(s):=\mathcal{F}' (w(s))^* [v- \mathcal{F}(w(s)) ], \label{eq:st:F}
     \\
      \mathrm{A}^\delta(s):= \hat{A}'(x^{\delta}(s))^* [v^{\delta} - \hat{A}(x^{\delta}(s)) ], & \quad \mathrm{A}(s):= \hat{A}' (w(s))^* [v - \hat{A}(w(s)) ]. \label{eq:st:A}
\end{align}
From \eqref{SAR:sol}, one can see that
\begin{align*}
    x^{\delta}(t) &\, -w(t) = \int_0^t \big(\mathrm{F}^\delta(s) - \mathrm{F}(s) \big)  ds 
     + \int_0^t \big( \lambda^{\delta}(s)   \mathrm{A}^\delta(s)  - \lambda(s) \mathrm{A}(s) \big) ds 
     + \int_0^t g^{\delta}(s)dB(s)
\end{align*}
which on using an elementary inequality $\|p+q+r\|^2\leq 3\|p\|^2+3\|q\|^2+3\|r\|^2$, for all $p,q,r \geq 0$ along with H\"older\textquotesingle{s} inequality, Theorem \ref{th:pre;stoc-int} and inequality \eqref{eq:MDP:f;del} gives
\begin{align}
    \mathbb E \big[ \|x^{\delta}(t)  -w(t)\|^2 \big] \leq &\, 3t  \int_0^t \mathbb E \big[ \|\mathrm{F}^\delta(s) - \mathrm{F}(s) \|^2 \big] ds \notag 
    \\
    & +  3t \int_0^t \mathbb E \big[ \|\lambda^{\delta}(s)   \mathrm{A}^\delta(s)  - \lambda(s) \mathrm{A}(s)\|^2 \big] ds + 3t \tau^2 \delta^2 \mu  \epsilon. \label{eq:di,xn,x}
\end{align}
Now, for the estimation of first term on the right-hand side of \eqref{eq:di,xn,x},  using (A1) and (A6) of Assumption \ref{Assum}, one obtains
    \begin{align}
       \mathbb E \big[\|\mathrm{F}^\delta(s) - \mathrm{F}(s) \|^2 \big]= &\,  \mathbb E \big[ \|\mathcal{F}' (x^{\delta}(s))^* [v^{\delta} - \mathcal{F}(x^{\delta}(s)) ] -\mathcal{F}' (w(s))^* [v- \mathcal{F}(w(s)) ] \|^2 \big] \notag
        \\
        \leq &\, 2 \mathbb E \big[ \|\mathcal{F}' (x^{\delta}(s))^* [v^{\delta} - \mathcal{F}(x^{\delta}(s)) ] -\mathcal{F}' (w(s))^* [v^{\delta} - \mathcal{F}(w(s)) ] \|^2 \big] \notag
        \\
        & + 2 \mathbb E \big[ \|\mathcal{F}' (w(s))^* [v^{\delta} - \mathcal{F}(w(s)) ] -\mathcal{F}' (w(s))^* [v- \mathcal{F}(w(s)) ] \|^2 \big] \notag
        \\
        \leq &\, 2 \tilde{L}_{\mathcal{F}}^2 \mathbb E \big[ \|x^{\delta}(s)-w(s)\|^2 \big]  + 2 L_{\mathcal{F}}^2\|v^\delta -v\|^2 . \label{eq:di,fn,f}
    \end{align}
Furthermore, for the estimation of second term in \eqref{eq:di,xn,x}, notice that
    \begin{align*}
       \mathbb E \big[\|\lambda^{\delta}(s)   \mathrm{A}^\delta(s)  - \lambda(s) \mathrm{A}(s)\|^2 \big] \leq 2 \mathbb E \big[ \|\lambda^{\delta}(s) (  \mathrm{A}^\delta(s) - \mathrm{A}(s))\|^2 \big] + 2 \mathbb E \big[ \|(\lambda^{\delta}(s)-\lambda(s)) \mathrm{A}(s)\|^2\big].
    \end{align*}
Now, on performing similar steps as in \eqref{eq:di,fn,f}, together with (A3) and (A6) of Assumption \ref{Assum}, one gets
\begin{align}
    \mathbb E \big[ \|\lambda^{\delta}(s) (  \mathrm{A}^\delta(s) - \mathrm{A}(s))\|^2 \big] = &\, 2 \Gamma^2 \tilde{L}_{\hat{A}}^2  \mathbb E \big[ \|x^{\delta}(s)-w(s)\|^2 \big] + 2 \Gamma^2  L_{\hat{A}}^2 \|v^\delta -v\|^2. \label{eq:An,f}
\end{align}
Utilizing \eqref{eq:di,fn,f} and \eqref{eq:An,f} into \eqref{eq:di,xn,x} yields
\begin{align*}
    \mathbb E \big[ \|x^{\delta}(t)  -w(t)\|^2 \big] \leq &\,  6t\big( \tilde{L}_\mathcal{F}^2 + \Gamma^2 \tilde{L}_{\hat{A}}^2 \big) \int_0^t \mathbb E \big[ \|x^{\delta}(s)-w(s)\|^2 \big]  ds + 6t \int_0^t |\lambda^{\delta}(s)-\lambda(s)|^2 \mathbb E \big[ \| \mathrm{A}(s)\|^2 \big] ds
    \\
    & + 6t \big(\tilde{L}_\mathcal{F}^2+ \Gamma^2 L_{\hat{A}}^2 \big) \|v^\delta-v\|^2 + 3t \tau^2 \delta^2 \mu \epsilon 
\end{align*}
which by the application of Gr\"onwall\textquotesingle{s} lemma gives
\begin{align*}
    \mathbb E \big[ \|x^{\delta}(t)  -w(t)\|^2 \big] \leq &\, e^{6t^2\big( \tilde{L}_\mathcal{F}^2 + \Gamma^2 \tilde{L}_{\hat{A}}^2 \big) } \Big[ 6t\int_0^t |\lambda^{\delta}(s)-\lambda(s)|^2 \mathbb E \big[ \| \mathrm{A}(s)\|^2 \big] ds + 6t \big(\tilde{L}_\mathcal{F}^2+ \Gamma^2 L_{\hat{A}}^2 \big) \|v^\delta -v\|^2 
    \\
    & + 3t \tau^2 \delta^2 \mu \epsilon \Big].
\end{align*}
Since $\lambda^{\delta}(s)\to \lambda(s)$ as $\delta \to 0$, so on using the dominated convergence, one concludes the desired result.
\end{proof}


\noindent Now, we establish the regularizing property of \texttt{D-STAR} method \ref{D-STAR} with \eqref{def:f} for the noisy data.

\begin{theorem}\label{th:main;SDP}
Let assumptions of Lemma~\ref{lem:mono}, \ref{lem:stab}, and Theorem~\ref{th:exa,conv} be satisfied. Further, assume $t^*$ denotes the stopping time associated with Rule~\ref{rule:MDP}. Then, there exists a solution $w^\dagger$ to \eqref{IP} such that
    \begin{align*}
        \lim_{\delta \to 0} \mathbb E \left[ \|x^{\delta}(t^*)-w^\dagger\|^2 \right] =0.
    \end{align*}
\end{theorem}

\begin{proof}
    Let $\delta_T$ be a sequence such that $\delta_T \to 0$ as $T \to \infty$ and consider noisy data $v^{\delta_T}$ associated with each $\delta_T$. For the pair $(\delta_T,v^{\delta_T})$, let $t^*_T:=t^*(\delta_T,v^{\delta_T})$ be the stopping time determined by Rule~\ref{rule:MDP}. Then, we need to consider the two cases: $ \sup_{T \geq 0} t^*_T < \infty$ and $ \lim_{T \to \infty } t^*_T \to \infty$. For the first case, result can be established using the arguments similar to  \cite[Theorem 2.4]{hanke1995convergence}. For the second case, we may assume that $t^*_T$ is a strictly monotonically increasing sequence with respect to $T$.  Consequently, for every $\bar{T}<T$, from Lemma \ref{lem:mono}, one can observe
    \begin{align*}
        \mathbb E\big[ \|w^{\delta_T}(t^*_T)-w^\dagger\|^2 \big] \leq  \mathbb E \big[ \|w^{\delta_T}(t^*_{\bar{T}})-w^\dagger\|^2 \big]
        \leq 2 \mathbb E \big[\|w^{\delta_T}(t^*_{\bar{T}})-w(t^*_{\bar{T}})\|^2 \big] + 2\mathbb E \big[\|w(t^*_{\bar{T}})-w^\dagger\|^2 \big].
    \end{align*}
    Since $\bar{T}$ is fixed, Lemma \ref{lem:stab} shows that the first term on the right-hand side converges to zero as $T\to\infty$. On the other hand, Theorem \ref{th:exa,conv} ensures that the second term can be made arbitrarily small upon choosing $\bar{T}$ sufficiently large. This concludes the desired result. 
\end{proof}


\section{Convergence Analysis of \texttt{D-STAR} Method under Heuristic Rule}\label{sec:heu}

\noindent 
This section is devoted for the convergence analysis of \texttt{D-STAR} method~\eqref{D-STAR} when the noise level $\delta$ is unknown. In this setting, we employ a heuristic stopping rule introduced in Rule~\ref{rule:HR} and establish the corresponding convergence results under this criterion.
We first describe the \texttt{D-STAR} method~\eqref{D-STAR} in detail and then prove its monotonicity and finite termination properties in Subsection~\ref{subsec:hr:method}. Subsequently, Subsection~\ref{sub:hr:ex} demonstrates its convergence to the true solution of \eqref{IP} for the exact data case. Finally, Subsection~\ref{sub:hr:noisy} investigates the regularizing property of the \texttt{D-STAR} method \eqref{D-STAR} for the noisy data case, in absence of noise level.

\subsection{The method}\label{subsec:hr:method}
For \texttt{D-STAR} with a heuristic stopping rule, we assume that Assumption \ref{Assum} is satisfied for $B_{\sqrt{10}\rho}(\tilde{w})$ and condition (A2) holds with $\mu \in (0,1/3)$, which arises from the mathematical constraints of the analysis.  
To facilitate the convergence analysis, we define $g^\delta(t,x^\delta(t))$ as
\begin{align}
    g^\delta(t,x^{\delta}(t)) = \sqrt{\mu \tilde{\epsilon}} \|\mathcal{F}' (x^{\delta}(t))^* [v^{\delta} - \mathcal{F}(x^{\delta}(t))] \| h(t), \quad \tilde{\epsilon} \in \Big(0, \frac{2-2\mu}{\mu L_{\mathcal{F}}^2} \Big), \label{def:f:hr}
\end{align}
where $h(t)$ is a monotonically decreasing function such that
  $\lim_{t \to \infty} h(t)=0$ {and}  $h(0)= 1.$ 
  \footnote{ When $\mathcal{F}$ is linear, i.e. $\mu =0$, no additional adjustment through $\epsilon$ is required and we set $g^\delta(t,x^{\delta}(t)) =  \|\mathcal{F}' (x^{\delta}(t))^* [v^{\delta} - \mathcal{F}(x^{\delta}(t))] \| h(t).$}
  We note $g^\delta(t,x^\delta(t))$ depends on $x^\delta(t)$, while independent of noise level $\delta$, and the condition on $\tilde{\epsilon}$ arises from $\tau$ introduced in \eqref{def:hat,t}.
Throughout this section, we adapt the definition of $g^\delta(t,x^\delta(t))$ given in~\eqref{def:f:hr}. We further assume that the regularization weight $\lambda^\delta(t)$ satisfies
$
\sup_{t\geq0}\lambda^\delta(t)=\tilde{\Gamma}:={R_{\mathcal{F}}^2}/{\tilde{\beta}},
$
where $\tilde{\beta}$ is a constant which depends on $v^\delta$, $L_{\hat{A}}$, $\rho$, $\tilde{w}$ and its explicit value is given in Lemma~\ref{l:Hr,mon}. Moreover, by (A5) of Assumption~\ref{Assum} and arguments analogous to Lemma~\ref{lem:well}, \texttt{D-STAR} method \eqref{D-STAR} with \eqref{def:f:hr} admits a unique mild solution.
Now, from (A1) of Assumption \ref{Assum}, one can observe that
\begin{align}
    \|g^\delta(t,x^{\delta}(t))\|^2 \leq {\mu \tilde{\epsilon}} \|\mathcal{F}' (x^{\delta}(t))^* [v^{\delta} - \mathcal{F}(x^{\delta}(t)) ] \|^2 |h(t)|^2 \leq \mu \tilde{\epsilon} L_{\mathcal{F}}^2 \|v^{\delta} - \mathcal{F}(x^{\delta}(t))\|^2  \label{eq:f:hr}
\end{align}
which will be a key component in subsequent analysis.

Before proceeding with the convergence results, we first prove the existence of finite stopping time $\tilde{t}^*$ defined by  Rule~\ref{rule:HR}. Indeed, if $t_\infty < \infty$ then its finiteness is immediate. Hence, it remains to consider the case $t_\infty = \infty$ only. In this case, from Rule \ref{rule:HR} and Assumption \ref{as:HR}, notice that
\begin{align*}
    \alpha(t,v^{\delta}) = (t+a) \mathbb E \big[ \|v^{\delta}-\mathcal{F}(x^{\delta}(t))\|^2 \big] \geq \kappa(t+a) \|v^{\delta}-v\|^2  \to \infty \quad \text{as} \quad  t \to \infty.
\end{align*}
Thus, there exists a finite $\tilde{t}^*$ which gives the infimum of $\alpha(t,v^{\delta})$.
Now, in order to establish the regularizing property of \texttt{D-STAR} method \eqref{D-STAR} with \eqref{def:f:hr}, one needs to show that $\alpha (\tilde{t}^*,v^{\delta}) \to 0$  \text{as} $\delta \to 0.$
For this purpose, we define an auxiliary stopping time $\hat{t}^*$ such that 
\begin{align}
    \mathbb E \big[ \|v^{\delta}- \mathcal{F}(x^\delta(\hat{t}^*))\|^2 \big] +\frac{\mathcal{C}}{(\hat{t}^* +a)^2} \leq \tau^2 \|v^\delta -v\|^2 < \mathbb E \big[ \|v^{\delta}- \mathcal{F}(x^\delta(t))\|^2 \big] +\frac{\mathcal{C}}{(t +a)^2} \label{def:hat,t}
\end{align}
for all $t < \hat{t}^*$, where $\mathcal{C}=\frac{3a \rho^2}{(\mu+1)}$ and $\tau > 1$ (distinct from $\tau$ introduced in Rule \ref{rule:MDP}). One can see that it is a  modification  of Rule \ref{rule:MDP} with an extra term $\mathcal{C}$. This extra term guarantees that  $\hat{t}^* \to \infty$ as $\delta \to 0$, which is essential to prove that $\alpha (\tilde{t}^*,v^{\delta}) \to 0$  \text{as} $\delta \to 0.$
The pseudo-code for \texttt{D-STAR} method \eqref{D-STAR} with \eqref{def:f:hr}, under Rule \ref{rule:HR} with corresponding parameters specified in this subsection is analogous to Algorithm~\ref{alg:MDP}.

\begin{lemma}\label{l:Hr,mon}
Suppose Assumptions \ref{Assum} and \ref{as:HR} hold and let $x^{\delta}(t)$ denotes the solution of  \texttt{D-STAR} method \eqref{D-STAR}. Further, assume that there exists a constant $\tilde{K}_{\lambda^\delta}$ such that
\begin{align}\label{as:lam:hr}
    \lambda^{\delta}(t) \leq &\, {\tilde{K}_{\lambda^{\delta}} \mathbb E \big[\|\mathcal{F}' (x^{\delta}(t))^* [v^{\delta}-\mathcal{F} (x^{\delta}(t))]\|^2 \big]}/ \tilde{\beta} \quad \text{and} \quad   \mathcal{H}_2:=  \mathcal{H}_{\tau} - \tilde{K}_{\lambda^{\delta}} L_{\mathcal{F}}^2 > 0,
\end{align}
where $\tilde{\beta}= \|v^{\delta}-\hat{A} (\tilde{w})\|^2+ 3{L}_{\hat{A}}^2 \rho^2/2$ and $\mathcal{H}_{\tau}:= 2\left( 1-\mu - \frac{\mu+1}{\tau} \right) - \mu \tilde{\epsilon} L_{\mathcal{F}}^2 >0$.
Then, $x^{\delta}(t) \in B_{\sqrt{10}\rho}(\tilde{w})$ for all $t< \hat{t}^*$, where $\hat{t}^*$ is defined in \eqref{def:hat,t}.
\end{lemma}

\begin{proof}
By applying the same arguments used in \eqref{eq:Ito:di x^n,x*} of Lemma \ref{lem:mono}, one has the following,
\begin{align}
       \mathbb E \big[\|x^{\delta}(t)- &\, w^\dagger \|^2 \big] =  \mathbb E \big[\| \tilde{w} -w^\dagger\|^2 \big] +  2\mathbb E\int_0^t \big[ \big \langle x^{\delta}(s)- w^\dagger,\mathcal{F}' (x^{\delta}(s))^* [v^{\delta} - \mathcal{F}(x^{\delta}(s)) ] \big \rangle \big] ds \notag 
       \\
       & +  2\mathbb E\int_0^t \big[ \big \langle x^{\delta}(s)- w^\dagger, \lambda^{\delta}(s) \hat{A}'(x^{\delta}(s))^* [v^{\delta}-\hat{A}(x^{\delta}(s))] \big \rangle \big] ds  +  \int_0^t \mathbb E \big[\|g^{\delta}(s,x^{\delta}(s))\|^2 \big] ds \notag 
       \\
      = &\, \| \tilde{w} -w^\dagger\|^2  +\tilde{\phi}_{\mathcal{F}} + \tilde{\phi}_{\hat{A}} + \tilde{\phi}_{{B}}, \label{eq:til-phi}
   \end{align}
  where $\tilde{\phi}_{\mathcal{F}}$ and $\tilde{\phi}_{\hat{A}}$ denote the terms associated with $\mathcal{F}$ and $\hat{A}$, respectively, while $\tilde{\phi}_B$ denotes the term associated with diffusion. To estimate $\tilde{\phi}_{\mathcal{F}}$, using similar arguments to \eqref{eq:phi-F}, one obtains
 \begin{align}\label{eq:til-phi-F*}
         \tilde{\phi}_{\mathcal{F}} := &\, 2 \int_0^t \mathbb E \big[ \big \langle w^{\delta}(s)- w^\dagger,\mathcal{F}' (w^{\delta}(s))^* [v^{\delta} - \mathcal{F}(x^{\delta}(s)) ] \big \rangle \big]ds  \notag
         \\
         \leq &\, \int_0^t \Big[ 2 (\mu-1) \mathbb E  \big [\|v^{\delta} - \mathcal{F}(w^{\delta}(s))\|^2 \big]  + 2 (\mu+1) \| v^{\delta} -v\| \mathbb E  \big [\|v^{\delta} - \mathcal{F}(w^{\delta}(s))\| \big] \Big] ds
     \end{align}
Now, for every $s <\hat{t}^*$, \eqref{def:hat,t} implies that 
\begin{align*}
    \|v^\delta -v\| < \frac{1}{\tau} \left( \mathbb E \big[ \|v^{\delta}- \mathcal{F}(x^\delta(s))\|^2 \big] +\frac{\mathcal{C}}{(s +a)^2} \right)^{1/2}
\end{align*}
which on using Young\textquotesingle{s} inequality yields
\begin{align}
     \| v^{\delta} -v\| \mathbb E  \big [\|v^{\delta} - \mathcal{F}(w^{\delta}(s))\| \big]  \leq \frac{1}{\tau} \left( \mathbb E  \big [\|v^{\delta} - \mathcal{F}(w^{\delta}(s))\|^2 \big]+ \frac{\mathcal{C}}{2(s +a)^2} \right). \label{eq:hr:ar:pro}
\end{align}
Thus, on utilizing \eqref{eq:til-phi-F*} and \eqref{eq:hr:ar:pro} {along with $\int_0^t \frac{1}{(s+a)^2}ds \leq \frac{1}{a}$}, one gets 
      \begin{align}
         \tilde{\phi}_{\mathcal{F}} \leq &\, 2\left( \mu-1 + \frac{\mu+1}{\tau} \right) \int_0^t \mathbb E  \big [\|v^{\delta} - \mathcal{F}(w^{\delta}(s))\|^2 \big] ds + \frac{\mathcal{C}(\mu+1)}{a}. \label{eq:til-phi-F}
     \end{align}
For the estimation of $\tilde{\phi}_{\hat{A}}$, using similar arguments to \eqref{eq:phi-A-int} along with  \eqref{as:lam:hr} and (A1), (A3) of Assumption~\ref{Assum}, one obtains
\begin{align} 
    \tilde{\phi}_{\hat{A}} := &\, 2 \mathbb E \int_0^t \lambda^{\delta}(s) \big[ \big \langle x^{\delta}(s)- w^\dagger,  \hat{A}'(x^{\delta}(s))^* [v^{\delta}-\hat{A}(x^{\delta}(s))] \big \rangle \big]ds 
    \leq \int_0^t  \frac{ \lambda^{\delta}(s)}{2} \|v^{\delta}-\hat{A}(w^\dagger)\|^2 ds \notag
    \\
    \leq &\, \int_0^t \lambda^{\delta}(s) \big[ \|v^{\delta}-\hat{A} (\tilde{w})\|^2 + \|\hat{A} (\tilde{w})-\hat{A}(w^\dagger)\|^2   \big] ds
    \leq \int_0^t \lambda^{\delta}(s)  \big( \|v^{\delta}-\hat{A} (\tilde{w})\|^2+ {L}_{\hat{A}}^2 \|w^\dagger-\tilde{w}\|^2  \big) ds\notag
     \\
      \leq &\, \tilde{K}_{\lambda^{\delta}} L_{\mathcal{F}}^2 \int_0^t \mathbb E \big[ \|v^{\delta}-\mathcal{F} (x^{\delta}(s))\|^2 \big] ds. \label{eq:til-phi-A}
\end{align}
Furthermore, by using \eqref{eq:f:hr}, $\tilde{\phi}_{B}$ can be estimated as
\begin{align}
    \tilde{\phi}_{B}:= \int_0^t \mathbb E \big[\|g^{\delta}(s,x^{\delta}(s))\|^2 \big] ds \leq \mu \tilde{\epsilon} L_{\mathcal{F}}^2 \int_0^t \mathbb E \big[ \|v^{\delta} - \mathcal{F}(x^{\delta}(s))\|^2  \big] ds. \label{eq:til-phi-W}
\end{align}
By substituting \eqref{eq:til-phi-F}, \eqref{eq:til-phi-A} and \eqref{eq:til-phi-W} into \eqref{eq:til-phi}, one gets
\begin{align*}
    \mathbb E \big[\|x^{\delta}(t)- &\, w^\dagger \|^2 \big] \leq  \| \tilde{w} -w^\dagger\|^2  - \left[2\left( 1-\mu - \frac{\mu+1}{\tau} \right) - \tilde{K}_{\lambda^{\delta}} L_{\mathcal{F}}^2 - \mu \tilde{\epsilon} L_{\mathcal{F}}^2 \right]\int_0^t  \mathbb E \big[ \|v^{\delta} - \mathcal{F}(x^{\delta}(s))\|^2  \big] ds + \frac{\mathcal{C}(\mu+1)}{a}
\end{align*}
which on putting $\mathcal{C}={3a \rho^2}/{(\mu+1)}$ gives
\begin{align}
    \mathbb E \big[\|x^{\delta}(t)- &\, w^\dagger \|^2 \big] + \mathcal{H}_2 \int_0^t  \mathbb E \big[ \|v^{\delta} - \mathcal{F}(x^{\delta}(s))\|^2  \big] ds \leq  \| \tilde{w} -w^\dagger\|^2  + \frac{\mathcal{C}(\mu+1)}{a} \leq 4 \rho^2. \label{eq:hr:er}
\end{align}
Thus, $x^{\delta}(t) \in B_{2\rho}(w^\dagger)$ for all $t< \hat{t}^*$. Since $\|\tilde{w}-w^\dagger\| \leq \rho$, therefore one obtains $\mathbb E \big[\|w^\delta(t) -\tilde{w} \|^2\big] \leq 2\mathbb E \big[\|w^\delta(t)-w^\dagger\|^2 \big] + 2\|w^\dagger-\tilde{w}\|^2 \leq 10 \rho^2$, which implies that $x^{\delta}(t) \in B_{\sqrt{10}\rho}(\tilde{w})$ for all $t< \hat{t}^*$. This completes the proof.
\end{proof}

\noindent Now, we show that there exists a finite stopping time $\hat{t}^*$ defined in \eqref{def:hat,t} and  $\hat{t}^* \to \infty$ as $\delta \to 0$.

\begin{lemma}\label{l:Hr:arb,t}
Suppose assumptions of Lemma \ref{l:Hr,mon} holds. Then, there exists a finite $\hat{t}^*$, as defined in \eqref{def:hat,t} such that
    $\hat{t}^* \to \infty$ as $\delta \to 0$.
\end{lemma}
 \begin{proof}
     Let us assume that there does not exist any finite $\hat{t}^*$. Then, for every $t \geq 0$, 
\(
        \tau^2 \|v^\delta -v\|^2 < \mathbb E \big[\|v^{\delta}- \mathcal{F}(x^{\delta}(t))\|^2 \big] +  \frac{\mathcal {C}}{(t +a)^2}, 
 \)
which on integrating from $0$ to $t<\hat{t}^*$ and using \eqref{eq:hr:er}, yields
\begin{align*}
   \tau^2 \|v^\delta -v\|^2 t <  \int_0^{t} \mathbb E \big[\|v^{\delta}- \mathcal{F}(x^{\delta}(s))\|^2 \big] ds + \frac{\mathcal{C}}{a} \leq \frac{4\rho^2}{\mathcal{H}_2} + \frac{\mathcal{C}}{a} < \infty.
\end{align*}
Taking $t \to \infty$ leads to a contradiction, and hence $\hat{t}^*$ must be finite. Moreover, from \eqref{def:hat,t}, one gets $\tau^2 \|v^\delta -v\|^2 \geq \frac{\mathcal {C}}{(\hat{t}^* +a)^2}.$ Since $\|v^\delta -v\|^2 \to 0$ as $\delta \to 0$, so $\hat{t}^* \to \infty$ as $\delta \to 0$. 
 \end{proof}
\noindent
We next establish the following lemma, which plays a crucial role in proving the regularizing property of \texttt{D-STAR} method \eqref{D-STAR}. The proof is motivated by the ideas developed in \cite{real2024hanke,zhang2018heuristic}.

\begin{lemma}\label{L:alp}
Suppose assumptions of Lemma \ref{l:Hr,mon} holds and $\tilde{t}^*$ denotes the stopping time associated with Rule \ref{rule:HR}. Then, $\alpha(\tilde{t}^*,v^{\delta}) \to 0$ as $\delta \to 0.$
Consequently, $\mathbb E \big[\|v^{\delta}-\mathcal{F}(x^\delta(\tilde{t}^*))\|^2 \big] \to 0 $ and $\tilde{t}^* \|v^\delta-v\|^2 \to 0$ as $\delta \to 0.$
\end{lemma}

\begin{proof}
    From Rule \ref{rule:HR}, one can observe that $ \mathbb E \big[\|v^{\delta}-\mathcal{F}(x^{\delta}(t))\|^2 \big] = \frac{\alpha(t,v^{\delta})}{t+a}\geq \frac{\alpha(\tilde{t}^*,v^{\delta})}{t+a}, $
    which on integrating over $0$ to $\hat{t}^*$ gives
    \begin{align*}
    \alpha(\tilde{t}^*,v^{\delta}) \log \Big( \frac{\hat{t}^* +a}{a} \Big) = \int_0^{\hat{t}^*} \frac{\alpha(\tilde{t}^*,v^{\delta})}{t+a} dt \leq \int_0^{\hat{t}^*} \mathbb E \big[\|v^{\delta}-\mathcal{F}(x^{\delta}(t))\|^2 \big] dt.
    \end{align*}
Hence,
\(
    \alpha(\tilde{t}^*,v^{\delta})  \leq \frac{1}{\log \Big( \frac{\hat{t}^* +a}{a} \Big)} \int_0^{\hat{t}^*} \mathbb E \big[\|v^{\delta}-\mathcal{F}(x^{\delta}(t))\|^2 \big] dt.
\)
Thus, on using \eqref{eq:hr:er} and Lemma \ref{l:Hr:arb,t} one obtains $\alpha(\tilde{t}^*,v^{\delta}) \to 0 $ as $\delta \to 0$.
 Furthermore, from Assumption \ref{as:HR}, one can conclude the desired result.
\end{proof}


\begin{lemma}\label{lem:hr:mono}
Suppose assumptions of Lemma \ref{l:Hr,mon} holds and $x^{\delta}(t)$ denotes the solution of  \texttt{D-STAR} method \eqref{D-STAR}. Let $\mathcal{H}_3:= 1-3\mu- \tilde{K}_{\lambda^{\delta}} L_{\mathcal{F}}^2- \mu \tilde{\epsilon} L_{\mathcal{F}}^2 >0$, where $1-3\mu-\mu \tilde{\epsilon} L_{\mathcal{F}}^2>0$ as  $\mu \in (0,1/3)$ and $\tilde{\epsilon} \in \Big(0, \frac{2-2\mu}{\mu L_{\mathcal{F}}^2} \Big)$.
Then, for sufficiently small $\delta$, $x^{\delta}(t) \in B_{\sqrt{10}\rho}(\tilde{w})$ for all $t\leq \tilde{t}^*$.
\end{lemma}

\begin{proof}
Recall \eqref{eq:til-phi} together with the definitions of $\tilde{\phi}_{\mathcal{F}}$, $\tilde{\phi}_{\hat{A}}$, and $\tilde{\phi}_{B}$ given therein. Now, we first estimate $\tilde{\phi}_{\mathcal{F}}$.  By using Young\textquotesingle{s} inequality, one gets
\(
 2\|v^{\delta} -v\| \mathbb E  \big [\|v^{\delta} - \mathcal{F}(w^{\delta}(s))\| \big] \leq \|v^{\delta} -v\|^2 +\mathbb E  \big [\|v^{\delta} - \mathcal{F}(w^{\delta}(s))\|^2 \big]. 
\)
Combining this inequality with \eqref{eq:til-phi-F*}, $\tilde{\phi}_{\mathcal{F}}$ can be estimated as
     \begin{align}
         \tilde{\phi}_{\mathcal{F}} := &\, 2 \int_0^t \mathbb E \big[ \big \langle w^{\delta}(s)- w^\dagger,\mathcal{F}' (w^{\delta}(t))^* [v^{\delta} - \mathcal{F}(x^{\delta}(t)) ] \big \rangle \big] \notag
         \\
         \leq &\, (3\mu-1) \int_0^t \mathbb E  \big [\|v^{\delta} - \mathcal{F}(w^{\delta}(s))\|^2 \big] ds +  (\mu+1) \| v^{\delta} -v\| t. \label{eq:til-phi-F:2}
     \end{align}
Thus, on utilizing \eqref{eq:til-phi} \eqref{eq:til-phi-A}, \eqref{eq:til-phi-W} and \eqref{eq:til-phi-F:2}, one obtains
\begin{align*}
     \mathbb E \big[\|w^{\delta}(t)- &\, w^\dagger \|^2 \big] =  \| \tilde{w} -w^\dagger\|^2 - \left(1-3\mu- \tilde{K}_{\lambda^{\delta}} L_{\mathcal{F}}^2- \mu \tilde{\epsilon} L_{\mathcal{F}}^2 \right) \int_0^t \mathbb E \big[ \|v^{\delta} - \mathcal{F}(x^{\delta}(t))\|^2  \big] ds  + (\mu+1) \| v^{\delta} -v\| t
\end{align*}
which implies that
\(
\mathbb E \big[\|w^{\delta}(t)- w^\dagger \|^2 \big] + \mathcal{H}_3  \int_0^t \|v^{\delta} - \mathcal{F}(x^{\delta}(t))\|^2 ds \leq \| \tilde{w} -w^\dagger\|^2  + (\mu+1) \| v^{\delta} -v\| \tilde{t}^*.
\)
By Lemma \ref{L:alp}, for sufficiently small $\delta$, $ (\mu+1)\|v^{\delta} -v\| \tilde{t}^* \leq 3\rho^2$. Consequently, $\mathbb E \big[\|w^{\delta}(t)- w^\dagger \|^2 \big] \leq 4\rho^2$, which concludes the proof.
\end{proof}


\subsection{Convergence for exact data}\label{sub:hr:ex}
Now, we prove the convergence of \texttt{D-STAR} method \eqref{D-STAR} with \eqref{def:f:hr} for the exact data $v$.
We emphasize that in the exact data case, the diffusion term vanishes in Rule \ref{rule:MDP} due to its explicit dependence on the noise level $\delta$, whereas it does not vanishes in Rule \ref{rule:HR}. Consequently, with this additional diffusion term, the pseudo-code for the convergence of \texttt{D-STAR} with \eqref{def:f:hr} for the exact data $v$ and the parameters specified in this subsection is analogous to Algorithm~\ref{alg:ex_data}.
The convergence analysis begins with the following lemma.

\begin{lemma}\label{lem:hr:exa:mono}
   Suppose assumptions of Lemma \ref{l:Hr,mon} holds and $w(t)$ denotes the solution of \texttt{D-STAR} method \eqref{D-STAR} with the exact data $v$. 
    Let there exists a constant $\tilde{K}_{\lambda}$ such that
\begin{align}
    \lambda(t) \leq &\, {\tilde{K}_{\lambda} \mathbb E \big [\|\mathcal{F}' (w(t))^* [v-\mathcal{F} (w(t))]\|^2 \big]}/\tilde{\gamma} \quad \text{and} \quad \mathcal{H}_4:= 2(1-  \mu) - \tilde{K}_{\lambda} L_{\mathcal{F}}^2-\mu \tilde{\epsilon} L_{\mathcal{F}}^2  >0,\label{eq:hr:lam:ex}
\end{align}
where $\tilde{\gamma} = \|v-\hat{A}(\tilde{w})\|^2 + 3{L}_{\hat{A}}^2 \rho^2/2 $
and $2(1-  \mu) -\mu \tilde{\epsilon} L_{\mathcal{F}}^2 >0$ as $\tilde{\epsilon} \in \Big(0, \frac{2-2\mu}{\mu L_{\mathcal{F}}^2} \Big)$.
Then, for every $t\geq 0$, the following holds,
    \begin{align}
        \frac{d}{dt} \mathbb E \big[\|w(t)- w^\dagger \|^2 \big] \leq &\, - \mathcal{H}_4 \mathbb E \big[ \|v-\mathcal{F} (w(t))\|^2 \big] \quad \text{and} \quad \int_0^\infty\mathbb E \big[ \|v  -\mathcal{F} (w(t))\|^2 \big]  dt  < \infty. \label{eq:hr:di:ex,x^n,x*}
    \end{align}
\end{lemma}

\begin{proof}
   The desired result follows from Lemma~\ref{l:Hr,mon} and fundamental theorem of calculus.
\end{proof}

\noindent
Since \texttt{D-STAR} method with the heuristic rule differs from the modified discrepancy principle due to the diffusion term, we need to establish the following convergence result for the exact data case separately.

\begin{theorem}
Let assumptions of Lemma \ref{lem:hr:exa:mono} be satisfied and $w(t)$ be the solution of \texttt{D-STAR} method \eqref{D-STAR} with the exact data $v$. Further, assume that $w^\dagger$ denotes a solution of \eqref{IP}. Then,
    \begin{align*}
       \lim_{t \to \infty} \mathbb E\big[ \| w(t)-w^\dagger\|^2 \big]  =0.
    \end{align*}
\end{theorem}

\begin{proof}
Let us define $\hat{e}(t) := w(t) - w^\dagger$ and we want to show that  $\{\hat{e}(t)\}_{t \geq 0}$ is a Cauchy sequence. Then, for any $r > t$, 
 \begin{align*}
     \mathbb E \big[\|\hat{e}(t)-\hat{e}(r)\|^2 \big] = \mathbb E \big[\|\hat{e}(t)\|^2 \big] - \mathbb E \big[\|\hat{e}(r)\|^2 \big] + 2 \mathbb E \big[ \langle \hat{e}(r)-\hat{e}(t),\hat{e}(r) \rangle \big].
 \end{align*}
Since Lemma \ref{lem:hr:exa:mono} establishes that both $\mathbb{E}[\Vert{}\hat{e}(t)\Vert{}^2]$ and $\mathbb{E}[\Vert{}\hat{e}(r)\Vert{}^2]$ decrease monotonically to some $\vartheta \geq 0$. Thus, it suffices to consider only the last term on the right-hand side of the above equation. Accordingly, for exact data $v$, using \eqref{D-STAR} together with \eqref{def:f:hr}, one gets
   \begin{align}
     \mathbb E  \big[   \langle \hat{e}(r) -\hat{e}(t),\hat{e}(r) \rangle \big] = &\,  
       \mathbb E \int_t^r  \big[ \langle v- \mathcal{F}(w(s)),\mathcal{F}' (w(s))\hat{e}(r)\rangle \big] ds + \mathbb E \int_t^r  \big[ \lambda(s) \langle  v-\hat{A}(w(s)), \hat{A}'(w(s))\hat{e}(r) \rangle \big] ds \notag
     \\
     & + \mathbb E \Big[ \Big \langle \int_t^r \| \mathcal{F}' (w(s))^* [v- \mathcal{F}(w(s)) ]\| h(s) dB(s), \hat{e}(r) \Big \rangle \Big]
     := \tilde{\psi}_{\mathcal{F}} + \tilde{\psi}_{\hat{A}} + \tilde{\psi}_B, \label{psi*}
   \end{align}    
    where $\tilde{\psi}_{\mathcal{F}}$, $\tilde{\psi}_{\hat{A}}$ represents the term associated with $\mathcal{F}$, $\hat{A}$, respectively, and $\tilde{\psi}_B$ denotes the term containing stochastic integral.
    By applying arguments similar to \eqref{eq:psi-F}, $\tilde{\psi}_{\mathcal{F}}$ can be estimated as 
    \begin{align}
        \tilde{\psi}_{\mathcal{F}}  := &\, \mathbb E \int_t^r \big[ \langle v- \mathcal{F}(w(s)),\mathcal{F}' (w(s))(w(r)-w^\dagger) \rangle \big] ds
        \leq  3(\mu+1) \int_t^r \mathbb E \big[ \|v-\mathcal{F}(w(s))\|^2 \big] ds. \label{eq:til:psi-F}
    \end{align}
   For the estimation of $\tilde{\psi}_{\hat{A}}$, by using (A3) of Assumption \ref{Assum} and H\"older\textquotesingle{s} inequality, one obtains
    \begin{align*}
         \tilde{\psi}_{\hat{A}} := &\, \mathbb E \int_t^r \big[ \lambda(s) \langle  v-\hat{A}(w(s)), \hat{A}'(w(s))\hat{e}(r) \rangle \big] ds 
         \leq  L_{\hat{A}} \int_t^r \lambda(s)  \big(\mathbb E \big[\|v-\hat{A}(w(s))\|^2 \big] \big)^{1/2} \big( \mathbb E \big[\|\hat{e}(r)\|^2 \big] \big)^{1/2} ds.
    \end{align*}
By triangle inequality and (A3) of Assumption \ref{Assum}, one gets 
\begin{align*}
    \|v-\hat{A}(w(s))\| \leq \|v-\hat{A} (\tilde{w})\| + \|\hat{A} (\tilde{w})-\hat{A}(w(s))\| \leq \|v-\hat{A} (\tilde{w})\|+ {L}_{\hat{A}} \|w(s)-\tilde{w}\|
\end{align*}
which on using  Young\textquotesingle{s} inequality and (A1) of Assumption \ref{Assum} together with \eqref{eq:hr:lam:ex}, yields
\begin{align} 
     \tilde{\psi}_{\hat{A}} \leq &\,   L_{\hat{A}} \int_t^r \lambda(s)  \big( 2\|v-\hat{A} (\tilde{w})\|^2+ 2{L}_{\hat{A}}^2\mathbb E \|w(s)-\tilde{w}\|^2 \big)^{1/2} \|w^\dagger-\tilde{w}\| ds \notag
      \\
      \leq &\, \int_t^r \lambda(s)  \big( \|v-\hat{A}(\tilde{w})\|^2 + 3{L}_{\hat{A}}^2 \|w^\dagger-\tilde{w}\|^2 /2 \big)ds 
      \leq  \tilde{K}_{\lambda^\delta} L_{\mathcal{F}}^2 \int_t^r  \mathbb E \|v- \mathcal{F}(w(s))\|^2 ds. \label{eq:til:psi-A}
\end{align}
Now, for the estimation of $\tilde{\psi}_B$, by using H\"older\textquotesingle{s} inequality, Theorem \ref{th:pre;stoc-int} along with \eqref{eq:hr:lam:ex} and (A1) of Assumption \ref{Assum}, 
   \begin{align}
     \tilde{\psi}_B := &\, \mathbb E \Big[ \Big \langle \int_t^r \|\mathcal{F}' (w(s))^* [v- \mathcal{F}(w(s))] \|h(s) dB(s), \hat{e}(r) \Big \rangle \Big] \notag
     \\ 
     \leq &\, \Big( \mathbb E \Big[  \Big \|\int_t^r \| \mathcal{F}' (w(s))^* [v- \mathcal{F}(w(s)) ]\| h(s)dB(s) \Big \|^2 \Big] \Big)^{1/2} \big( \mathbb E  \big[ \|\hat{e}(r)\|^2 \big] \big)^{1/2} \notag
       \\
       \leq &\, \Big( \mathbb E \Big[  \int_t^r \|\mathcal{F}' (w(s))^* [v- \mathcal{F}(w(s)) ]\|^2 ds \Big] \Big)^{1/2}  \big( \mathbb E \big[ \|\hat{e}(r)\|^2 \big] \big)^{1/2}  \notag
       \\
       \leq &\,
       L_{\mathcal{F}} \Big(  \int_t^r \hspace{-0.1cm} \mathbb E  \big[\| v- \mathcal{F}(w(s))\|^2 \big] ds \Big)^{1/2} \hspace{-0.09cm} \|w^\dagger-\tilde{w}\|. \label{eq:til:psi-W} 
   \end{align}
 Thus, on utilizing \eqref{eq:til:psi-F}, \eqref{eq:til:psi-A} and \eqref{eq:til:psi-W} into \eqref{psi*}, one gets
 \begin{align*}
     \mathbb E \big[ \langle \hat{e}(r)-\hat{e}(t),\hat{e}(r) \rangle \big] \leq &\, \big( 3(\mu+1) + \tilde{K}_{\lambda^\delta} L_{\mathcal{F}}^2  \big) \int_t^r \mathbb E \big[ \|v-\mathcal{F}(w(s))\|^2 \big] ds 
     + L_{\mathcal{F}} \Big(  \int_t^r \mathbb E  \| v- \mathcal{F}(w(s))\|^2 ds \Big)^{1/2}  \|w^\dagger-\tilde{w}\|
 \end{align*}
which on using \eqref{eq:hr:di:ex,x^n,x*} yields
\(
      \lim_{t \to \infty} \mathbb E \big[ \langle \hat{e}(r)-\hat{e}(t),\hat{e}(r) \rangle \big] =0.
\)
Hence, $\{\hat{e}(t)\}_{t \ge 0}$ is a Cauchy sequence. Consequently, $\{w(t)\}_{t \ge 0}$ is also a Cauchy sequence and converges to $w^\dagger$, which completes the proof.
\end{proof}

\begin{remark}
    By applying  Jensen\textquotesingle{s} inequality, one obtains $ \| \mathbb E[w(t)]-w^\dagger\|^2 \leq \mathbb E \big[\|w(t)-w^\dagger \|^2 \big]$, which gives \( \lim_{t \to \infty}  \| \mathbb E[w(t)]-w^\dagger\|^2  =0.\)  Hence, one gets $\mathbb E[w(t)] \to w^\dagger$ as $t \to \infty$.
\end{remark}

\subsection{Convergence for noisy data}\label{sub:hr:noisy}
We now present the main result of this section, which establishes the regularizing property of \texttt{D-STAR} method \eqref{D-STAR} with \eqref{def:f:hr} for noisy data $v^\delta$ under  Rule~\ref{rule:HR}. Before that, we prove the following stability lemma.
\begin{lemma}\label{lem:stab:Heu}
Suppose Assumptions \ref{Assum}, \ref{as:HR} hold and $\lambda^{\delta}(t) \to \lambda(t)$ as $\delta \to 0$ for any fixed $t$. Further, Let $x^\delta(t)$ be the solution of \texttt{D-STAR} method \eqref{D-STAR} and $w(t)$ denotes its solution with exact data $v$. Then, 
    \begin{align*}
        \lim_{\delta \to 0 } \mathbb E \big [\|x^{\delta}(t)-w(t)\|^2 \big] =0.
    \end{align*}
\end{lemma}

\begin{proof}
Recall $\mathrm{F}^\delta(s), \mathrm{F}{(s)}$ and $\mathrm{A}^\delta(s), \mathrm{A}(s)$ from \eqref{eq:st:F}.
Then, from \eqref{D-STAR} along with \eqref{def:f:hr}, 
\begin{align*}
    x^{\delta}(t) -w(t) = &\, \int_0^t \big(\mathrm{F}^\delta(s) - \mathrm{F}(s) \big)  ds 
     + \int_0^t \big( \lambda^{\delta}(s)   \mathrm{A}^\delta(s)  - \lambda(s) \mathrm{A}(s) \big) ds +
     \sqrt{\mu \tilde{\epsilon}} \int_0^t h(s) \big( \|\mathrm{F}^\delta(s)\| - \|\mathrm{F}(s)\| \big) dB(s)
\end{align*}
which on using an elementary inequality $\|p+q+r\|^2\leq 3\|p\|^2+3\|q\|^2+3\|r\|^2$, for all $p,q,r \geq 0$, along with H\"older\textquotesingle{s} inequality and Theorem \ref{th:pre;stoc-int}, yields
    \begin{align*}
    \mathbb E \big[ \|x^{\delta}(t)  -w(t)\|^2 \big] \leq &\, 3t  \int_0^t \mathbb E \big[ \|\mathrm{F}^\delta(s) - \mathrm{F}(s) \|^2 \big] ds \notag 
    \\
    & +  3t \int_0^t \mathbb E \big[ \|\lambda^{\delta}(s)   \mathrm{A}^\delta(s)  - \lambda(s) \mathrm{A}(s)\|^2 \big] ds   + 3{\mu \tilde{\epsilon}} \int_0^t |h(s)|^2 \mathbb E \big[\| \mathrm{F}^\delta(s) - \mathrm{F}(s)\|^2 \big] ds
    \\
    \leq &\, 3(t +\mu \tilde{\epsilon})  \int_0^t \mathbb E \big[ \|\mathrm{F}^\delta(s) - \mathrm{F}(s) \|^2 \big] ds +  3t \int_0^t \mathbb E \big[ \|\lambda^{\delta}(s)   \mathrm{A}^\delta(s)  - \lambda(s) \mathrm{A}(s)\|^2 \big] ds. 
\end{align*}
Now, on using \eqref{eq:di,fn,f} and  \eqref{eq:An,f}, one obtains
\begin{align*}
    \mathbb E \big[ \|x^{\delta}(t)  -w(t)\|^2 \big] \leq &\, 6 \big( (t + \mu \tilde{\epsilon})\tilde{L}_\mathcal{F}^2+ t\Gamma^2 L_{\hat{A}}^2 \big) \big) \int_0^t \mathbb E \big[ \|x^{\delta}(s)-w(s)\|^2 \big]  ds 
    \\
    & + 6t\int_0^t (\lambda^{\delta}(s)-\lambda(s)) \mathbb E \big[ \| \mathrm{A}(s)\|^2 \big] ds
   + 6\big( (t + \mu \tilde{\epsilon})\tilde{L}_\mathcal{F}^2+ t\Gamma^2 L_{\hat{A}}^2 \big)  \|v^\delta-v\|^2.
\end{align*}

Since $\lambda^\delta(s) \to \lambda (s)$ as $\delta \to 0$, therefore by applying Gr\"onwall\textquotesingle{s} lemma and dominated convergence theorem, one obtains the  result.
\end{proof}
\begin{theorem}
Let assumptions of Lemma \ref{l:Hr,mon}, \ref{lem:stab:Heu} and Theorem \ref{sub:hr:ex} be satisfied. Assume $\tilde{t}^*$ be the stopping time defined in Rule  \ref{rule:HR}. Then, there exists a solution $w^\dagger$ to \eqref{IP} such that
    \begin{align*}
        \lim_{\delta \to 0} \mathbb E \big[ \|x^\delta(\tilde{t}^*)-w^\dagger\|^2 \big] =0.
    \end{align*}
\end{theorem}

\begin{proof}
   Assume $\delta_T$ be a sequence such that $\delta_T \to 0$ as $T \to \infty$ and ${v}^{\delta_T}$ denotes the corresponding noisy data sequence. For the pair 
   $(\delta_T,v^{\delta_T})$, let $\tilde{t}^*_T:= \tilde{t}^*(v^{\delta_T})$ denotes the stopping time given in Rule \ref{rule:HR}. Then, we need to analyze two cases: $ \sup_{T \geq 0} \tilde{t}^*_T < \infty$ and $ \lim_{T \to \infty } \tilde{t}^*_T \to \infty$.  For the first case, consider a convergent subsequence $\{\tilde{t}^*_{T}\}_{T\geq 0}$ such that $\tilde{t}^*_{T} \to \tilde{t}$ as $T \to \infty$. 
    Lemma \ref{lem:stab:Heu} along with continuity of $\mathcal{F}$, implies that
   \(
        \lim_{T \to \infty} \mathbb E \big[\|\mathcal{F}(w(\tilde{t}))- \mathcal{F}(w^{\delta_T}(\tilde{t}))\|^2 \big] =0
\) 
Moreover,
    \begin{align*}
      \mathbb E \big[\|\mathcal{F}(w(\tilde{t}))-v\|^2 \big] \leq &\, 3 \mathbb E \big[ \|\mathcal{F}(w(\tilde{t}))- \mathcal{F}(w^{\delta_T}(\tilde{t}))\|^2 \big] + 3 \mathbb E \big[ \|\mathcal{F}(w^{\delta_T}(\tilde{t}))-v^{\delta_T}\|^2 \big] +   3 \|v^{\delta_T}-v\|^2.
    \end{align*}
  which on taking $T \to \infty$ and using Lemma \ref{L:alp} gives $\mathcal{F}(w(\tilde{t}))= v$. Therefore, by the uniqueness of the solution and Lemma \ref{lem:stab:Heu}, one obtains
$\mathbb E \big[\|x^\delta(\tilde{t})  - w^\dagger \|^2 \big] \to 0 $ as $T \to \infty$.
     Moreover, for the second case, i.e. $\lim_{T \to \infty } \tilde{t}^*_T \to \infty$, the proof is analogous to Theorem \ref{th:main;SDP} and is omitted for brevity. 
\end{proof}


\section{Numerical Experiments}\label{sec:num}

\noindent This section presents the numerical experiments to illustrate the performance of \texttt{D-STAR} method under the stopping rules given in Rule \ref{rule:MDP} and \ref{rule:HR}. We first examine a linear inverse problem governed by the Radon transform, then a nonlinear inverse problem involving the Schlieren operator. In both cases, \texttt{D-STAR} is compared with the existing methods \texttt{DDIRLI}~\eqref{DDLI} and \texttt{SAR} \eqref{SAR}, which can be view as special cases of \eqref{D-STAR}.
In \texttt{D-STAR}, both the data-driven drift and the diffusion terms are included. In contrast, \texttt{SAR} is obtained by setting \(\lambda(t)=0\), while \texttt{DDIRLI} for discrete setting is obtained by setting \(h(t)=0\).
Since \texttt{D-STAR} and \texttt{SAR} are stochastic methods, multiple sample paths can be generated. Thus, to ensure a fair comparison  with \texttt{DDIRLI}, we use a single sample path. On the other hand, when comparing \texttt{D-STAR} with \texttt{SAR}, we generate 24 sample paths for each method throughout this section. 

We now turn to the numerical implementation of \texttt{D-STAR} method \eqref{D-STAR}. Since it is formulated as a stochastic differential equation in the artificial time $t$, a suitable time discretization is required. Although several numerical schemes are available for approximation, see ~\cite{jentzen2009numerical,lord2014introduction}. But for simplicity, we consider the following Euler method
\begin{align}
    w_{k+1}^\delta = w_k^\delta + \Delta t \big( \mathcal{F}'(w_k^\delta)^* (v^\delta-\mathcal{F}(w_k^\delta)) + \lambda_k^\delta \hat{A}'(w_k^\delta)^* (v^\delta-\hat{A}(w_k^\delta)) \big) +g^\delta (t_k,w_k^\delta) \Delta B_k, \quad k \in \mathbb N, \label{dis:D_STAR}
\end{align}
where $\Delta t$ is the size of uniform time step and $\Delta B_k = \sqrt{\Delta t}\zeta_k$ with $\{\zeta_k\}_{k \geq 0}$ being a sequence of independent standard Gaussian random variables. 
Moreover, to obtain the noisy data $v^\delta$, we introduce a Gaussian noise $\theta$ to the exact data $v$ with a noise level $\delta$. Specifically, 
\(
    v^\delta = v + \delta_{rel} \|v\| \frac{\theta}{\|\theta\|},
\)
where $\delta_{rel}=\|v^\delta-v\|/\|v\|$ denotes the relative noise level, while $\delta=\delta_{rel} \|v\|$ is the corresponding noise level. 

Throughout the numerical experiments, $k^*$ and $\hat{k}^*$ denote the stopping indices of Rule \ref{rule:MDP} and Rule \ref{rule:HR}, respectively, while their corresponding relative errors at the stopping indices are given by $e^*= \|w^\delta_{k^*} -w^\dagger\|/ \|w^\dagger\|$ and $\hat{e}^*= \|w^\delta_{\hat{k}^*} -w^\dagger\|/ \|w^\dagger\|$.

\subsection{X-ray Computed Tomography(CT)}
X-ray CT is an important medical imaging technique which is used for detecting tumors, bone fractures, and other injuries, see ~\cite{quinto2006introduction} for further details. 
In the discrete setting, it is modeled by a linear system
\(
    Rw= v^\delta,
\)
where $R \in \mathbb R^{n \times m}$ denotes discretized Radon transform matrix, $w \in \mathbb R^{n}$ represents the unknown two-dimensional image, whose pixels are lexicographically arranged using row-major or column-major ordering and $v^\delta \in \mathbb R^{m}$ denotes corresponding noisy sinogram.

In the numerical experiment, we consider
 $180$ parallel-beams with equi-spaced projection angles within $[0,\pi)$. The unknown true image is discretized on a uniform $64\times64$ pixel grid, giving $n=4096$, while each projection has $182$
detector bins, resulting in a sinogram of size $182\times180$, giving $m=32760$, which leads to an underdetermined ill-posed system. The data-driven operator $\hat{A}$ is constructed using a training dataset of $64 \times 64$ pixel images of handwritten digits from the MNIST database~\cite{lecun1998gradient}, following the procedure described in~\cite{aspri2020data}.
Specifically, we utilize $50,600,$ and $2000$ grayscale images for training in Test 1(a), 1(b) and 1(c), respectively.
In {Test 1(a)}, the target image is selected from the training set, whereas {Test 1(b)} uses an image from the validation set. Finally, {Test 1(c)}, reconstructs the Shepp–Logan phantom, which is completely different from the MNIST dataset. For the Shepp--Logan phantom, we additionally run \texttt{SAR} and \texttt{D-STAR} to determine  uncertainty quantification of the solution. 

Starting with an initial value $\tilde{w}=0$ and the maximum iterations as $600$, we consider two distinct parameter configurations. 
Under Rule \ref{rule:MDP}, we define $h(t)=\tau/(1+t)$ and $\tau=1.3$. The regularization parameter $\lambda_k$ is set as $6.6062 \times 10^{-8} \|v^\delta-\mathcal{F}(w_k^\delta)\|^2$, $5.8275 \times 10^{-8} \|v^\delta-\mathcal{F}(w_k^\delta)\|^2$, and  $1.4883 \times 10^{-7} \|v^\delta-\mathcal{F}(w_k^\delta)\|^2$, for Tests 1(a), 1(b), and 1(c), respectively. The resulting reconstructed solutions obtained via \texttt{DDIRLI}, \texttt{SAR}, and \texttt{D-STAR} methods, alongside their corresponding relative error and residual graphs, are illustrated in Figures~\ref{fig:MDP_CT_rec}-~\ref{MDP_CT_residual}. Furthermore, for Rule \ref{rule:HR}, we set
$h(t)=1/(1+t)$ and $a=10$. The regularization parameter $\lambda_k$ is set as $1.2068 \times 10^{-5} \|v^\delta-\mathcal{F}(w_k^\delta)\|^2$, $9.2356 \times 10^{-5} \|v^\delta-\mathcal{F}(w_k^\delta)\|^2$, and  $8.8435 \times 10^{-5} \|v^\delta-\mathcal{F}(w_k^\delta)\|^2$, for Tests 1(a), 1(b), and 1(c), respectively. Figures ~\ref{HR_CT_rec}-~\ref{HR_CT_res}, represents the reconstructed solutions obtained via all three methods for all three tests. Moreover, a detailed comparison of relative errors and iterations under both rules with different choices of $\delta$, $\tau$, and $a$, respectively, is given in Table~\ref{tab:dp-hr-summary}.

\begin{figure}[htbp]
  \centering
  \includegraphics[width=0.7\textwidth]{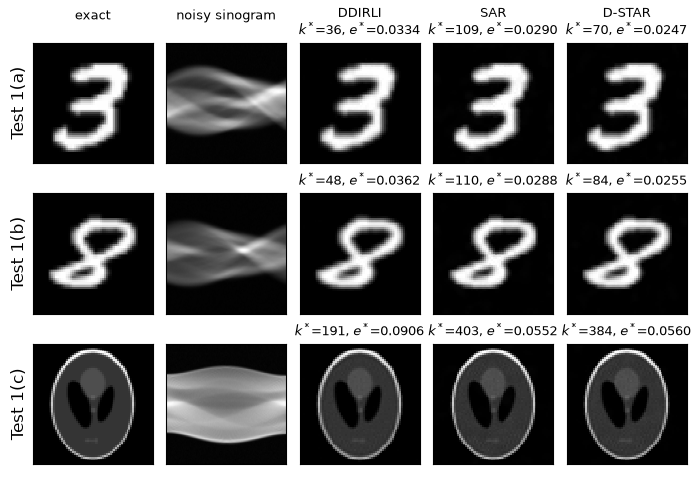}
  \vspace{-0.5cm}
  \caption{\scriptsize Reconstructed solutions of X-ray CT under Rule \ref{rule:MDP} with $\delta=1\%$ and $\tau=1.3$ for Tests 1(a), 1(b) and 1(c).}
  \label{fig:MDP_CT_rec}
\end{figure}

\begin{figure}[htbp]
  \centering
  \includegraphics[width=0.7\textwidth]{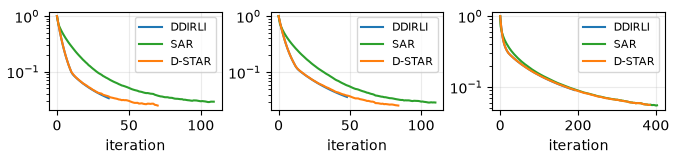}
  \vspace{-0.3cm}
  \caption{\scriptsize Relative error of X-ray CT under Rule \ref{rule:MDP} with $\delta=1\%$ and $\tau=1.3$ for Tests 1(a), 1(b) and 1(c), left to right.}
  \label{MDP_CT_error}
\end{figure}

\begin{figure}[htbp]
  \centering
  \includegraphics[width=0.7\textwidth]{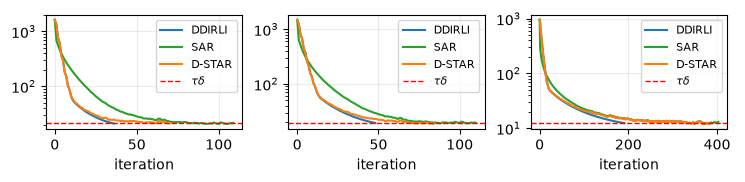}
  \vspace{-0.3cm}
  \caption
 {\scriptsize Residual of X-ray CT under Rule \ref{rule:MDP} with $\delta=1\%$ and $\tau=1.3$ for Tests 1(a), 1(b) and 1(c), left to right.
}
  \label{MDP_CT_residual}
\end{figure}
 
\begin{figure}[htbp]
  \centering
  \includegraphics[width=0.7\textwidth]{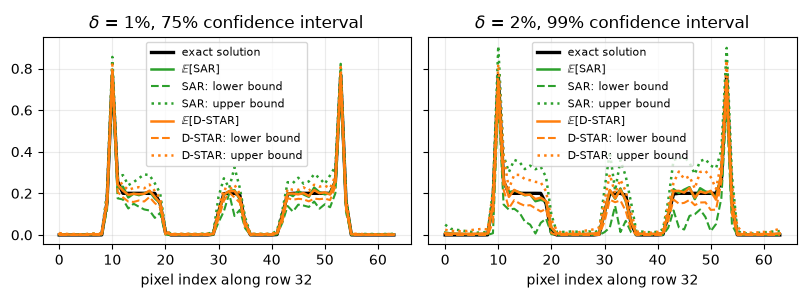}
  \vspace{-0.3cm}
  \caption{\scriptsize Uncertainty quantification of Shepp--Logan phantom for X-ray CT under Rule \ref{rule:MDP} with $\tau=1.3$ and $24$ sample paths.}
  \label{MDP_CT_UQ}
\end{figure}

We observe from Figure \ref{fig:MDP_CT_rec} that \texttt{D-STAR} method under Rule \ref{rule:MDP} provides better reconstruction accuracy compared to \texttt{DDIRLI} and \texttt{SAR} in Tests 1(a) and 1(b). Furthermore, it converges in fewer iterations than \texttt{SAR}, though it requires more iterations than \texttt{DDIRLI} due to the inclusion of an extra diffusion term. For Test 1(c), \texttt{D-STAR} achieves a significantly smaller error than \texttt{DDIRLI} (at cost of almost twice iterations than \texttt{DDIRLI}), while performing almost similar to \texttt{SAR} in reconstruction quality (\texttt{D-STAR} is requiring less iterations than \texttt{SAR}).
Figure \ref{HR_CT_rec} shows that under Rule \ref{rule:HR}, \texttt{D-STAR} and \texttt{DDIRLI} have almost same accuracy for Tests 1(a) and 1(b), although \texttt{D-STAR} is requiring more iterations than \texttt{DDIRLI}. Both methods achieve smaller errors with fewer iterations than \texttt{SAR}. In Test 1(c), \texttt{D-STAR} and \texttt{DDIRLI} perform in same way, whereas \texttt{SAR} reduces the error at the cost of more iterations.
Figures \ref{MDP_CT_error}, \ref{HR_CT_err} and \ref{MDP_CT_residual}, \ref{HR_CT_res} present the relative error and residual curves of \texttt{DDIRLI}, \texttt{SAR}, and \texttt{D-STAR} for a single sample path across Tests 1(a), 1(b), and 1(c), under Rules \ref{rule:MDP} and \ref{rule:HR}, respectively. These graphs demonstrate that the relative error and residuals decrease most slowly for \texttt{SAR}, followed by \texttt{D-STAR}, and most rapidly for \texttt{DDIRLI} under both rules. Furthermore, uncertainty quantification of the Shepp-Logan phantom under Rules \ref{rule:MDP} and \ref{rule:HR} are given in Figures \ref{MDP_CT_UQ} and \ref{HR_CT_UQ.png}, respectively. The band-width and iterations required for \texttt{SAR} and \texttt{D-STAR}  with different noise level and confidence interval are given in Table \ref{tab:sar-dstar-uq}. Across both settings, \texttt{D-STAR} yields sharper bands at substantially smaller iterations (see Table \ref{tab:sar-dstar-uq}).


\begin{figure}[htbp]
\small
  \centering
  \includegraphics[width=0.7\textwidth]{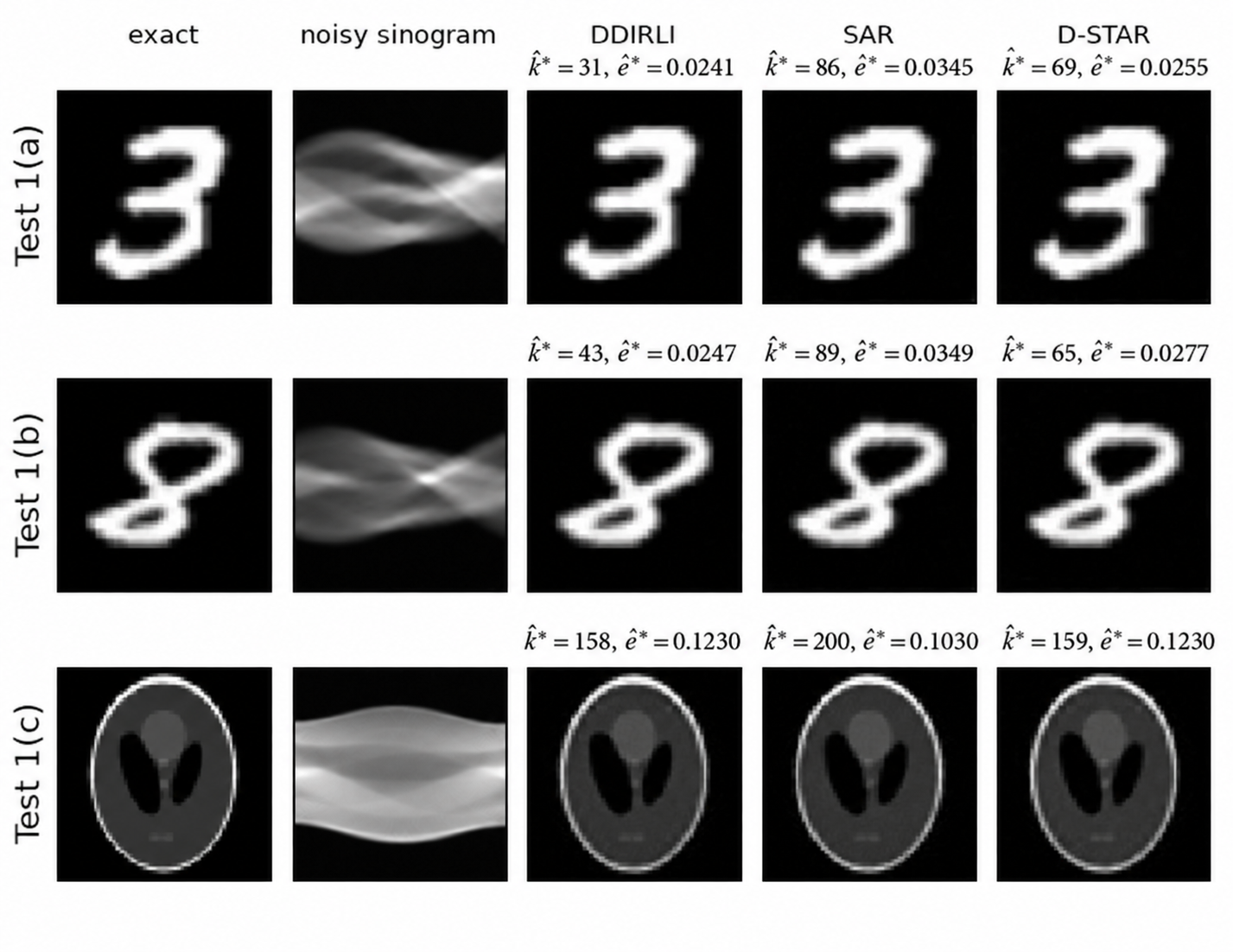}
  \vspace{-0.5cm}
 \caption{\scriptsize Reconstructed solutions of X-ray CT under Rule \ref{rule:HR} with $\delta=1\%$ and $a=10$ for Tests 1(a), 1(b) and 1(c).}
  \label{HR_CT_rec}
\end{figure}
 
\begin{figure}[htbp]
  \centering
  \includegraphics[width=0.7\textwidth]{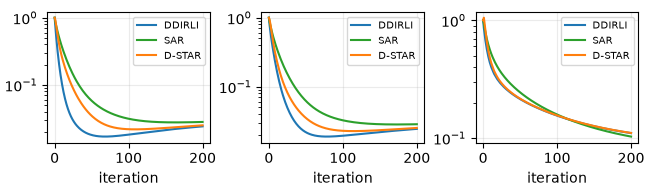}
  \vspace{-0.3cm}
  \caption{\scriptsize Relative error of X-ray CT under Rule \ref{rule:HR} with $\delta=1\%$ and $a=10$ for Tests 1(a), 1(b) and 1(c), left to right.}
  \label{HR_CT_err}
\end{figure}

\begin{figure}[htbp]
  \centering
  \includegraphics[width=0.7\textwidth]{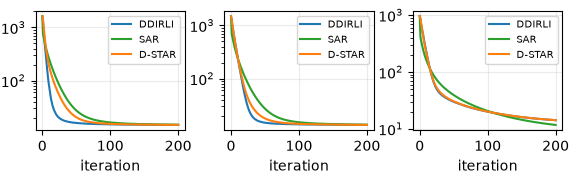}
  \vspace{-0.2cm}
  \caption
 {\scriptsize Residual of X-ray CT under Rule \ref{rule:HR} with $\delta=1\%$ and $a=10$ for Tests 1(a), 1(b) and 1(c), left to right.
}
  \label{HR_CT_res}
\end{figure}

\begin{figure}[htbp]
\small
  \centering
  \includegraphics[width=0.55\textwidth]{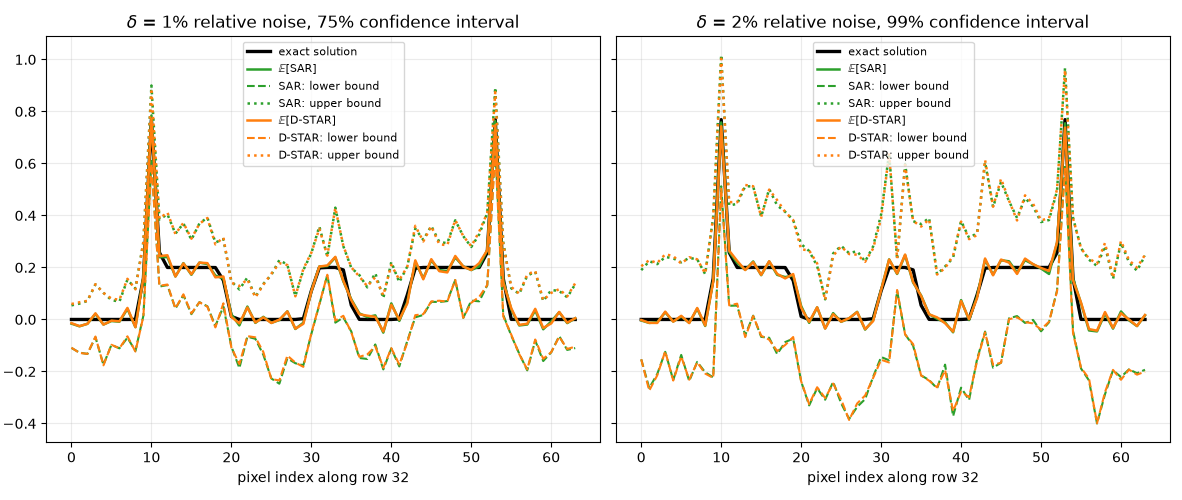}
  \vspace{-0.3cm}
  \caption{\scriptsize Uncertainty quantification of Shepp--Logan phantom for X-ray CT under Rule \ref{rule:HR} with $a=10$ and $24$ sample paths.}
  \label{HR_CT_UQ.png}
\end{figure}


\begin{table}[htbp]
  \centering
  \begin{tabular}{cclcccc}
    \toprule
    & & & \multicolumn{2}{c}{Rule \ref{rule:MDP}} & \multicolumn{2}{c}{Rule \ref{rule:HR}} \\
    \cmidrule(lr){4-5} \cmidrule(lr){6-7}
    Noise & Confidence interval & Method
      & Band width & $k^*$
      & Band width & $\hat{k}^*$ \\
    \midrule
    \multirow{2}{*}{$1\%$} & \multirow{2}{*}{$75\%$}
      & \texttt{SAR}    & $0.0527$ & $1076$ & $0.2365$ & $600$ \\
    & & \texttt{D-STAR} & $0.0307$ & $613$  & $0.2366$ & $573$ \\
    \midrule
    \multirow{2}{*}{$2\%$} & \multirow{2}{*}{$99\%$}
      & \texttt{SAR}    & $0.1183$ & $562$  & $0.4414$ & $451$ \\
    & & \texttt{D-STAR} & $0.0749$ & $321$  & $0.4484$ & $408$ \\
    \bottomrule
  \end{tabular}
  \caption{\scriptsize Uncertainty quantification of Shepp--Logan phantom for X-ray CT using \texttt{SAR} and \texttt{D-STAR} for $24$ sample paths under Rule \ref{rule:MDP} with $\tau=1.3$ and Rule \ref{rule:HR} with $a=10$.}
  \label{tab:sar-dstar-uq}
\end{table}


\begin{table}[htbp]
\small
\setlength{\aboverulesep}{0pt}
\setlength{\belowrulesep}{0pt}
\renewcommand{\arraystretch}{1.2}
\resizebox{\textwidth}{!}{%
\begin{tabular}{l|l|rcrcrcr|rcrcrcr}
\toprule
 & & \multicolumn{7}{c|}{Modified discrepancy principle} & \multicolumn{7}{c}{Heuristic rule} \\
\cmidrule(lr){3-9}\cmidrule(lr){10-16}
 & & & \multicolumn{2}{c}{$1\%$} & \multicolumn{2}{c}{$2\%$} & \multicolumn{2}{c|}{$3\%$}
 & & \multicolumn{2}{c}{$1\%$} & \multicolumn{2}{c}{$2\%$} & \multicolumn{2}{c}{$3\%$} \\
\cmidrule(lr){4-5}\cmidrule(lr){6-7}\cmidrule(lr){8-9}
\cmidrule(lr){11-12}\cmidrule(lr){13-14}\cmidrule(lr){15-16}
Test & Method & $\tau$ & $e^*$ & $k^*$ & $e^*$ & $k^*$ & $e^*$ & $k^*$
 & $a$ & $\hat{e}^*$ & $\hat{k}^*$ & $\hat{e}^*$ & $\hat{k}^*$ & $\hat{e}^*$ & $\hat{k}^*$ \\
\midrule
\multirow{9}{*}{1(a)} & \multirow{3}{*}{DDIRLI}
   & 1.3 & 0.0334 & 36 & 0.0547 & 21 & 0.0719 & 15 & 1   & 0.0274 & 27 & 0.0420 & 19 & 0.0496 & 17 \\
 & & 1.5 & 0.0408 & 29 & 0.0681 & 16 & 0.0901 & 11 & 10  & 0.0241 & 31 & 0.0375 & 21 & 0.0436 & 19 \\
 & & 2.0 & 0.0543 & 21 & 0.0898 & 11 & 0.0987 & 10 & 100 & 0.0185 & 47 & 0.0269 & 32 & 0.0333 & 27 \\
\cmidrule(lr){2-16}
 & \multirow{3}{*}{SAR}
   & 1.3 & 0.0290 & 109 & 0.0480 & 70 & 0.0646 & 61 & 1   & 0.0358 & 82  & 0.0534 & 58 & 0.0694 & 47 \\
 & & 1.5 & 0.0391 & 69  & 0.0659 & 47 & 0.0870 & 37 & 10  & 0.0345 & 86  & 0.0520 & 60 & 0.0670 & 49 \\
 & & 2.0 & 0.0563 & 49  & 0.0914 & 33 & 0.1207 & 27 & 100 & 0.0305 & 106 & 0.0448 & 77 & 0.0569 & 64 \\
\cmidrule(lr){2-16}
 & \multirow{3}{*}{D-STAR}
   & 1.3 & 0.0247 & 70 & 0.0421 & 37 & 0.0607 & 24 & 1   & 0.0260 & 67 & 0.0415 & 47 & 0.0572 & 37 \\
 & & 1.5 & 0.0336 & 42 & 0.0576 & 23 & 0.0800 & 15 & 10  & 0.0255 & 69 & 0.0394 & 50 & 0.0531 & 40 \\
 & & 2.0 & 0.0459 & 27 & 0.0830 & 14 & 0.1153 & 10 & 100 & 0.0228 & 85 & 0.0346 & 64 & 0.0448 & 54 \\
\midrule
\multirow{9}{*}{1(b)} & \multirow{3}{*}{DDIRLI}
   & 1.3 & 0.0362 & 48 & 0.0587 & 30 & 0.0810 & 21 & 1   & 0.0259 & 41 & 0.0388 & 31 & 0.0487 & 27 \\
 & & 1.5 & 0.0438 & 40 & 0.0749 & 23 & 0.0999 & 16 & 10  & 0.0247 & 43 & 0.0364 & 33 & 0.0453 & 29 \\
 & & 2.0 & 0.0606 & 29 & 0.1046 & 15 & 0.1266 & 12 & 100 & 0.0208 & 54 & 0.0307 & 41 & 0.0386 & 37 \\
\cmidrule(lr){2-16}
 & \multirow{3}{*}{SAR}
   & 1.3 & 0.0288 & 110 & 0.0456 & 75 & 0.0580 & 68 & 1   & 0.0354 & 87  & 0.0523 & 62 & 0.0686 & 50 \\
 & & 1.5 & 0.0375 & 72  & 0.0635 & 51 & 0.0874 & 42 & 10  & 0.0349 & 89  & 0.0504 & 65 & 0.0652 & 53 \\
 & & 2.0 & 0.0529 & 54  & 0.0935 & 37 & 0.1295 & 29 & 100 & 0.0313 & 108 & 0.0448 & 80 & 0.0564 & 68 \\
\cmidrule(lr){2-16}
 & \multirow{3}{*}{D-STAR}
   & 1.3 & 0.0255 & 84 & 0.0459 & 50 & 0.0624 & 33 & 1   & 0.0287 & 62 & 0.0462 & 42 & 0.0626 & 33 \\
 & & 1.5 & 0.0351 & 56 & 0.0614 & 32 & 0.0861 & 22 & 10  & 0.0277 & 65 & 0.0435 & 45 & 0.0578 & 36 \\
 & & 2.0 & 0.0524 & 36 & 0.0901 & 20 & 0.1321 & 13 & 100 & 0.0243 & 81 & 0.0365 & 60 & 0.0472 & 50 \\
\midrule
\multirow{9}{*}{1(c)} & \multirow{3}{*}{DDIRLI}
   & 1.3 & 0.0906 & 191 & 0.1447 & 104 & 0.1900 & 67 & 1   & 0.1275 & 147 & 0.2293 & 43  & 0.2561 & 33 \\
 & & 1.5 & 0.1073 & 154 & 0.1729 & 78  & 0.2283 & 46 & 10  & 0.1230 & 158 & 0.1926 & 65  & 0.2351 & 41 \\
 & & 2.0 & 0.1418 & 105 & 0.2298 & 45  & 0.2986 & 22 & 100 & 0.1078 & 211 & 0.1355 & 141 & 0.1647 & 98 \\
\cmidrule(lr){2-16}
 & \multirow{3}{*}{SAR}
   & 1.3 & 0.0552 & 403 & 0.0849 & 264 & 0.1152 & 194 & 1   & 0.0947 & 236 & 0.1406 & 126 & 0.1882 & 80  \\
 & & 1.5 & 0.0761 & 246 & 0.1199 & 151 & 0.1607 & 109 & 10  & 0.0936 & 200 & 0.1358 & 134 & 0.1771 & 89  \\
 & & 2.0 & 0.1090 & 161 & 0.1753 & 89  & 0.2349 & 60  & 100 & 0.0879 & 280 & 0.1180 & 176 & 0.1452 & 130 \\
\cmidrule(lr){2-16}
 & \multirow{3}{*}{D-STAR}
   & 1.3 & 0.0560 & 384 & 0.0925 & 221 & 0.1254 & 151 & 1   & 0.1271 & 149 & 0.2053 & 57  & 0.2429 & 39  \\
 & & 1.5 & 0.0785 & 234 & 0.1269 & 136 & 0.1723 & 84  & 10  & 0.1230 & 159 & 0.1839 & 73  & 0.2249 & 47  \\
 & & 2.0 & 0.1108 & 150 & 0.1879 & 70  & 0.2597 & 35  & 100 & 0.1081 & 211 & 0.1351 & 143 & 0.1640 & 100 \\
\bottomrule
\end{tabular}%
}
\caption{\scriptsize Numerical comparison of \texttt{DDIRLI}, \texttt{SAR}, and \texttt{D-STAR} for X-ray CT under Rule \ref{rule:MDP} and Rule \ref{rule:HR} for Tests 1(a), 1(b), and 1(c) at $\delta=1\%$, $2\%$, and $3\%$.}
\label{tab:dp-hr-summary}
\end{table}

Furthermore, from Table \ref{tab:dp-hr-summary}, one can observe that under Rule \ref{rule:MDP}, the relative error increases, while the number of iterations decreases across all three schemes and tests as both $\tau$ and $\delta$ increase, respectively. Further, under Rule \ref{rule:HR}, for any fixed $\delta$, the error decreases and iterations increases as $a$ increases for all three schemes and tests. However, for any given value of $a$, as $\delta$ increases, the relative error is increasing while number of iterations are decreasing across all three schemes and tests.

\subsection{Schlieren Tomography}
We investigate a Schlieren tomography problem in which the measured light-intensity pattern is proportional to  square of line integral of the pressure along the optical path.
A detailed description of the Schlieren tomography can be found in \cite{haltmeier2007kaczmarz,pitts1994tomographic}.

For $j \in 0,\ldots,l-1$, let $\nu_j = (\cos \theta_j, \sin \theta_j) \in S^1$ be the sets of recording directions, where $\theta_j \in [0,\pi)$ and $S^1$ represents the surface of the unit disc in $\mathbb R^2$. Then, we want to recover a function $w:\mathcal{D} \to \mathbb R$, where $\mathcal{D}=\{w \in \mathbb R^2: |w|<1\}$, from the following system of equations
\(
\mathcal{F}_j (w) = v^\delta, \quad j \in \{0,\ldots,l-1 \},
\)
where the Schlieren transform is given by $ \mathcal{F}_j = R_j^2$ with $R_j: C^\infty_0(\mathcal{D}) \to C^\infty_0([-1,1])$ denotes the Radon transform in the direction of $\nu_j$. 
According to \cite[Theorem 4.1]{haltmeier2007kaczmarz}, for each $j\in{0,\ldots,l-1}$, the operator
$\mathcal{F}_j$ is continuous and Fr\'echet differentiable from $H_0^1(\mathcal{D}) \to L^2([-1,1])$ with derivative given by $ \mathcal{F}'_i(w)h = 2R_i(w)R_i(h),$
where $H_0^1(\mathcal{D})$ denotes the Sobolev space of square-integrable functions on $\mathcal{D}$ with square-integrable weak gradients and vanishing trace on $\partial \mathcal{D}$. Furthermore, as established in \cite[Theorem 4.2]{haltmeier2007kaczmarz}, the adjoint operator $\mathcal{F}'_j(w)^*:L^2([-1,1]) \to H_0^1(\mathcal{D})$ is given by $g \to \mathcal{F}_j'(w)^*(g) = f(g)$, where $f(g)$ is an unique solution of $(I-\Delta)f(g) = 2R_j^*(R_j(w)g)$. Here  $I$ denotes the identity operator and $\Delta$ represents the Laplace operator on $H_0^1(\mathcal{D})$.

We mainly perform three numerical tests: 2(a), 2(b), and 2(c). The required $110 \times 110$ pixel training images for these tests are taken from Figures~14 and~15, respectively, as presented in \cite{aspri2020data}.
In {Test 2(a)} we show the reconstruction of an image only with region having positive values from training set. In {Test 2(b)} we consider an image only with a region having  negative values from training set. In {Test 2(c)} we reconstruct an image with a region with positive and negative values from training set.

For the numerical experiments under Rule \ref{rule:MDP}, we consider
 $180$ parallel-beams with equi-spaced projection angles within $[0,\pi)$. We set an initial value $\tilde{w}=0.01$, function $h(t)=\tau/(1+t)$, parameters $\tau=2$ and $\delta=5\%$. The regularization parameter $\lambda_k$ is set as $3.8360 \times 10^{-5} \|v^\delta-\mathcal{F}(w_k^\delta)\|^2$, $4.7230 \times 10^{-5} \|v^\delta-\mathcal{F}(w_k^\delta)\|^2$, and  $7.8244 \times 10^{-5} \|v^\delta-\mathcal{F}(w_k^\delta)\|^2$, for Tests 2(a), 2(b), and 2(c), respectively. The resulting reconstructed solutions obtained via the \texttt{DDIRLI}, \texttt{SAR}, and \texttt{D-STAR} methods, along with their corresponding relative error and residual curves, are presented in Figure \ref{fig:schlieren:rec}-\ref{fig:MDP_SCH_res}.


\begin{figure}[p]
  \centering
  \includegraphics[width=0.6\textwidth]{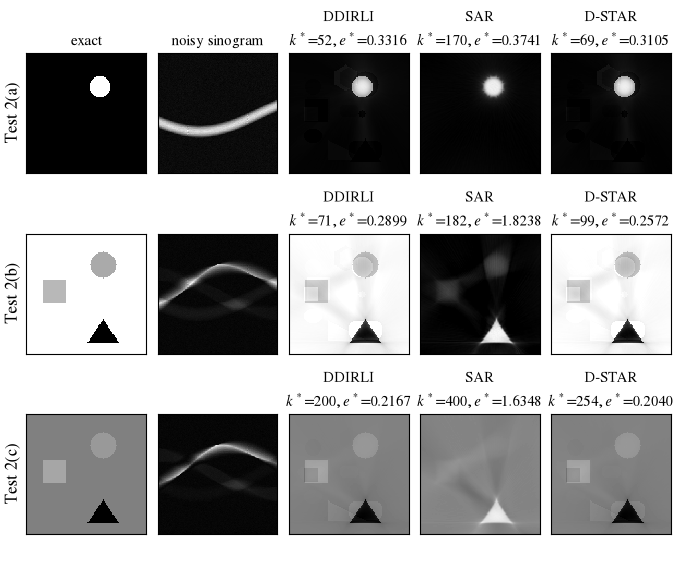}
 \vspace{-0.5cm}
  \caption{\scriptsize Reconstructed solutions of Schlieren tomography under Rule \ref{rule:MDP} with $\delta=5\%$ and $\tau=2$ for Tests 2(a), 2(b) and 2(c).}
  \label{fig:schlieren:rec}
\end{figure}
 
\begin{figure}[t]
  \centering
 \includegraphics[width=0.6\textwidth]{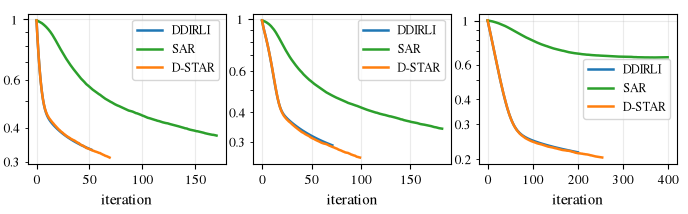}
  \vspace{-0.3cm}
  \caption{\scriptsize Relative error of Schlieren tomography under Rule \ref{rule:MDP} with $\delta=5\%$ and $\tau=2$ for Tests 2(a), 2(b) and 2(c), left to right.}
  \label{fig:schlieren:err}
\end{figure}

\begin{figure}[t]
    \centering
    \resizebox{0.6\textwidth}{3.3cm}{%
     \includegraphics{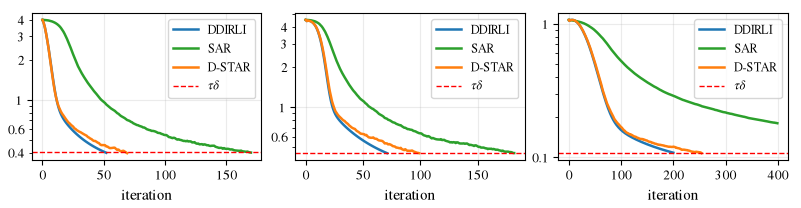}%
    }
    \vspace{-0.3cm}
    \caption{\scriptsize Residual of Schlieren tomography under Rule
    \ref{rule:MDP} with $\delta=5\%$ and $\tau=2$ for Tests 2(a), 2(b), and 2(c), from left to right.}
    \label{fig:MDP_SCH_res}
\end{figure}

\begin{figure}[t]
  \centering
\includegraphics[width=0.6\textwidth]{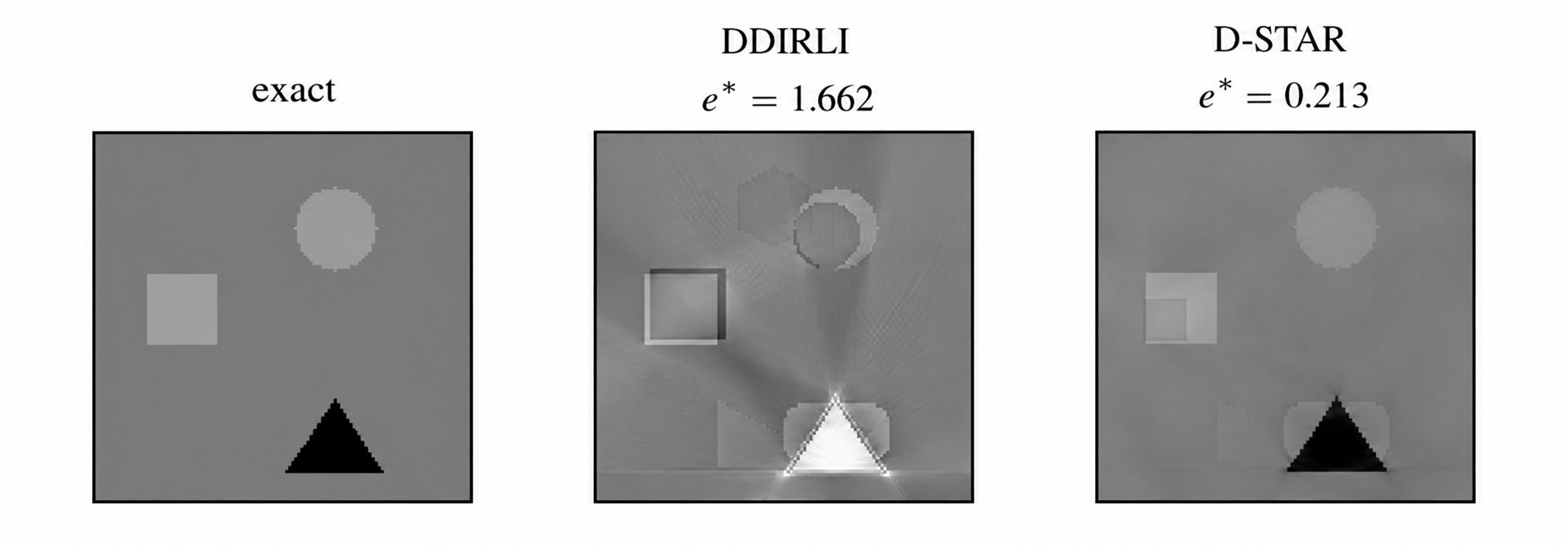}
 \vspace{-0.3cm}
  \caption
  {\scriptsize Exact solution (left), \texttt{DDIRLI} reconstruction (middle), and \texttt{D-STAR} reconstruction (right) under Rule \ref{rule:MDP} with $\delta=5\%$ and $\tau=2$.
}
  \label{fig:schlieren:mirror}
\end{figure}

\label{tab:schlieren}
\begin{table}[htbp]
  \centering
  \small
  \begin{tabular}{llccrc}
    \toprule
    Test & Method & Noise & $\tau$ & $k^*$ &
      $e^*$ \\
    \midrule
    2(a) & \texttt{DDIRLI} & 5\% & 2 &  52 & 0.3316 \\
         & \texttt{SAR}    &     &      & 170 & 0.3741 \\
         & \texttt{D-STAR} &     &      &  69 & {0.3105} \\
    \midrule
    2(b) & \texttt{DDIRLI} & 5\% & 2 &  71 & 0.2899 \\
         & \texttt{SAR}    &     &      & 182 & 1.8238 \\
         & \texttt{D-STAR} &     &      &  99 & {0.2572} \\
    \midrule
    2(c) & \texttt{DDIRLI} & 5\% & 2 & 200 & 0.2167 \\
         & \texttt{SAR}    &     &      & $400$ & 1.6348 \\
         & \texttt{D-STAR} &     &      & 254 & {0.2040} \\
    \bottomrule
  \end{tabular}
  \caption{\scriptsize Numerical comparison of \texttt{D-STAR} for Schlieren tomography under Rule \ref{rule:MDP} for Tests 2(a), 2(b) and 2(c)}
\end{table}

Figure \ref{fig:schlieren:rec} shows that \texttt{D-STAR} achieves better reconstruction accuracy than both \texttt{DDIRLI} and \texttt{SAR} in Tests 2(a), 2(b), and 2(c), while requiring fewer iterations than \texttt{SAR} but more than \texttt{DDIRLI}. 
Figures \ref{fig:schlieren:err} and \ref{fig:MDP_SCH_res} present the corresponding relative error and residual curves of \texttt{DDIRLI}, \texttt{SAR}, and \texttt{D-STAR} for a single path across all tests. Figure \ref{fig:schlieren:err} shows that error graph of \texttt{DDIRLI} and \texttt{D-STAR} are identical and decreasing more rapidly than \texttt{SAR} in all three Tests, while Figure \ref{fig:MDP_SCH_res} illustrates that residuals decrease most slowly for \texttt{SAR}, followed by \texttt{D-STAR}, and most rapidly for \texttt{DDIRLI}. Furthermore, Figure \ref{fig:schlieren:mirror} shows that \texttt{DDIRLI} gets trapped in a local minima, whereas the stochastic diffusion in \texttt{D-STAR} enables it to escape.


\section{Concluding remarks} \label{sec:con}
\noindent
In this work, we developed the Data-driven STochastic Asymptotic Regularization (\texttt{D-STAR}) method for solving both linear and nonlinear ill-posed inverse problems under known and unknown noise levels. The proposed method combines the mechanism of data-driven Iteratively Regularized Landweber Iteration (\texttt{IRLI}) with Stochastic Asymptotic Regularization (\texttt{SAR}). For known noise level, we employ a modified discrepancy principle, while a heuristic stopping rule is proposed for an unknown noise level. We considered two distinct stochastic diffusion terms in each case due to the definitions of stopping rules and mathematical challenges, and established the well-posedness of corresponding \texttt{D-STAR} method. Furthermore, we prove the finite termination of proposed method, convergence to the true solution for the exact data case, and the regularizing property of \texttt{D-STAR} method under both stopping rules. Numerical experiments involving X-ray CT and Schlieren tomography demonstrate the effectiveness of \texttt{D-STAR} method. We have given a detailed comparison of \texttt{D-STAR} with \texttt{DDIRLI} and \texttt{SAR}. The stochastic formulation enables \texttt{D-STAR} to quantify uncertainty in the computed solution and avoid local minima entrapment, whereas the data-driven component significantly reduces number of iterations compared to \texttt{SAR}.

The present analysis is formulated in a Hilbert space framework, while extending the proposed methodology to Banach space setting remains an interesting future work. Another possible direction is to derive explicit convergence rates for \texttt{D-STAR} method, following the ideas presented in \cite{zhang2024stochastic,zhou2025data}. It would also be interesting to develop a graph Laplacian based regularization method \cite{aleotti2026iterated, bajpai2026graph,bajpai2026convergence,bianchi2025data} in stochastic differential equation framework as such methods can effectively capture the structural information in image data and may provide a promising approach to regularization in continuous-time settings, which is currently not present in the literature.
\\

\noindent
\textbf{Data availability:} The authors will provide the datasets and source code supporting the findings of this study upon reasonable request.
\\

\noindent \textbf{Acknowledgment:} The authors used Claude as a coding support tool for debugging and verifying the code developed in the numerical experiments.
\\

\vspace{-0.4cm}






\bibliography{ref.bib}
\bibliographystyle{abbrv}

\end{document}